%% file: LS_Modes_Reconstruction.tex
\documentclass[final,10pt,preprint]{aimdyn-rep}

\usepackage{datetime}
\newdateformat{specialdate}{{\shortmonthname} \THEDAY, \THEYEAR}

\renewcommand\reportyear{2018}  
\renewcommand\reportmonthcreated{\shortmonthname[11]}	
\renewcommand\reportnumber{001} 
\renewcommand\thisversion{1}  
\renewcommand\revisiondate{\specialdate\today}
\renewcommand\thiscontract{DARPA HR0011-16-C-0116}  

\usepackage{algorithmic}
\usepackage{algorithm}
\usepackage[]{mcode}
\usepackage{float}
\usepackage{relsize}
\usepackage{tikz-cd}

\usepackage{tikz}
\usepackage{verbatim}
\usepackage[margin=1in]{geometry}
\usepackage{graphicx}  
\usepackage{subcaption}  
\usepackage{enumerate}
\usepackage{paralist}
\usepackage{xcolor} 
\usepackage[hypcap]{caption}
\usepackage{url}
\usepackage{amsfonts,amssymb,amsbsy,amsmath,amsthm}
\usepackage{mathrsfs}  
\usepackage{mathtools}
\usepackage{bm} 	
\usepackage{bbm} 	

\newcommand{\R}{ \ensuremath{\mathbb{R}}}  
\newcommand{\C}{ \ensuremath{\mathbb{C}}}  

\DeclarePairedDelimiterX\set[1]{\lbrace}{\rbrace}{  #1 }  

\DeclarePairedDelimiterX\norm[1]{\lVert}{\rVert}{#1}  			
\DeclarePairedDelimiterX\inner[2]{\langle}{\rangle}{#1 \,,\, #2}  	
\DeclareMathOperator{\diag}{diag}								

\DeclarePairedDelimiterX\abs[1]{\lvert}{\rvert}{#1} 

\newcommand{\0}{\mathbf{0}}
\newcommand{\Id}{\mathbb I}

\newtheorem{theorem}{Theorem}[section]

\newtheorem{proposition}[theorem]{Proposition}
\newtheorem{corollary}[theorem]{Corollary}

\newtheorem{remark}[theorem]{\textit{Remark}}
\newtheorem{example}[theorem]{\textit{Example}}
\newcommand{\A}{\mathbb{A}}

\newcommand{\eb}{\mathbf{e}}
\newcommand{\X}{\mathbf{X}}
\newcommand{\Y}{\mathbf{Y}}

\newcommand{\f}{\mathbf{f}}
\newcommand{\g}{\mathbf{g}}
\newcommand{\h}{\mathbf{h}}
\newcommand{\x}{\mathbf{x}}
\newcommand{\y}{\mathbf{y}}
\newcommand{\ww}{\mathfrak{w}}
\newcommand{\WW}{\mathbf{W}}
\newcommand{\Vanderm}{\mathbb{V}}

\newcommand{\ii}{\mathfrak{i}}

\newcommand{\roff}{{\boldsymbol\varepsilon}}
\newcommand{\bfalpha}{{\boldsymbol\alpha}}
\newcommand{\bfbeta}{{\boldsymbol\beta}}
\newcommand{\bfrho}{{\boldsymbol\rho}}
\newcommand{\bfLambda}{{\boldsymbol\Lambda}}
\newcommand{\bfnu}{{\boldsymbol\nu}}
\newcommand\restr[2]{{
		\left.\kern-\nulldelimiterspace 
		#1 
		\vphantom{\big|} 
		\right|_{#2} 
	}}

\usepackage{hyperref}

\allowdisplaybreaks

\begin{document}
	
	
	\title[LS reconstruction]{On least squares problems with certain Vandermonde--Khatri--Rao structure with applications to DMD}
	
	
	\author[Drma\v{c}, Mezi\'{c} and Mohr]{Zlatko Drma\v{c}\affil{1}\comma\corrauth and Igor Mezi\'{c}\affil{2} and Ryan Mohr\affil{3}}
	\address{\affilnum{1}\ Faculty of Science, Department of Mathematics, University of Zagreb, 10000 Zagreb, Croatia.\\
	\affilnum{2}\ Department of Mechanical Engineering and Mathematics, University of California, Santa Barbara, CA
	93106, USA \\
	\affilnum{3}\ AIMdyn, Inc., Santa Barbara, CA 93101, USA}
	%
	%
	\emails{{\tt drmac@math.hr} (Z. Drma\v{c}), {\tt mezic@ucsb.edu} (I. Mezi\'{c}), {\tt mohrr@aimdyn.com} (R. Mohr)}
	%
	
\begin{abstract}
This paper proposes a new computational method for solving structured least squares problems that arise in the process of identification of coherent structures in fluid flows. It is deployed in combination with dynamic mode decomposition (DMD) which provides a non-orthogonal set of modes --- corresponding to particular temporal frequencies --- a subset of which is used to represent time snapshots of the underlying dynamics. The coefficients of the representation are determined from a solution of a structured linear least squares problem with the matrix that involves the Khatri--Rao product of a triangular and a Vandermonde matrix. Such a structure allows a very efficient normal equation based least squares solution, which is used in state of the art CFD tools such as the sparsity promoting DMD (DMDSP). A new numerical analysis of the normal equations approach provides insights about its applicability and its limitations. Relevant condition numbers that determine numerical robustness are identified and discussed. Further, the paper offers a corrected semi-normal solution and QR factorization based algorithms. It is shown how to use the Vandermonde--Khatri--Rao structure to efficiently compute the QR factorization of the least squares coefficient matrix, thus providing a new computational tool for the ill-conditioned cases where the normal equations may fail to compute a sufficiently accurate solution. Altogether, the presented material provides a firm numerical linear algebra framework for a class of structured least squares problems arising in a variety of applications.
 \\
\end{abstract}
%
%
	\keywords{coherent structure, dynamic mode decomposition, Khatri--Rao product, Koopman operator, Krylov subspaces, proper orthogonal decomposition, Rayleigh-Ritz approximation, structured least squares, Vandermonde matrix}
%
	\ams{15A12, 15A23, 65F35, 65L05, 65M20, 65M22, 93A15, 93A30, 93B18, 93B40, 93B60, 93C05, 93C10, 93C15, 93C20, 93C57}
%
	\maketitle
	


\input{SOURCES2/Introduction_LS}
\input{SOURCES2/Section_Sparse_Reconstruction}
\input{SOURCES2/Section_Reconstruction_Algorithms}

\input{SOURCES2/LS_Concluding_Remarks}
\input{SOURCES2/Section_Appendix_2}
\bibliography{SOURCES2/Reconstruction_references1}
\bibliographystyle{plain}

\end{document}

%% file: SOURCES2/Introduction_LS.tex
\section{Introduction}
\noindent Dynamic mode decomposition (DMD) \cite{Schmid2011} has become a tool of trade in computational fluid dynamics (CFD), both in the  high fidelity numerical simulations and in the pure data driven scenarios;
for a review see \cite{Modal-analysis-fluid-flow-overview-2017} and references therein.
Its data driven framework and
deep theoretical roots in the Koopman operator theory  \cite{rowley2009spectral}, \cite{Williams2015}, make DMD a versatile tool in a plethora of applications beyond CFD, e.g. for studying dynamics of infectious diseases \cite{Proctor-DMD-infectious-2015} or for revealing spontaneous subtle emotions on human faces \cite{2016arXiv160104805C} in the field of affective computing. 

A distinctive feature of the DMD is that it is purely data driven. It does not assume knowledge of the solution of the governing, generally nonlinear, equations obeyed by the dynamics under study; it does assume that the supplied data snapshots (vectors of observables) $\f_i\in\mathbb{C}^n$,  are generated by a linear operator $\A$ so that $\f_{i+1}=\A \f_i$, $i=1,\ldots, m$, with some initial $\f_1$, and a constant time lag $\delta t$. One can think of $\A$ as e.g. a black-boxed numerical simulation  software, or e.g. as a discretized physics taken by a high--speed camera such as in studying  flame dynamics and combustion instabilities \cite{Combustion-Inst-Flame-Images-2016}.
%
The spectral data for $\A$ are then inferred by a combination of the proper orthogonal decomposition (POD) and the Rayleigh--Ritz projection. Under certain conditions, DMD can serve as an approximation to the Koopman operator underlying the process evolution \cite{rowley2009spectral,Arbabi:2017fm,Korda:2017jl}. For a good approximation in the sense of Koopman operator, the key is that the set of observables is well-selected and rich enough -- this is a separate issue, not considered in this paper.

The computed Ritz pairs $(\lambda_j, z_j)$ are  used for spatial--temporal decomposition of the snapshots. More precisely, 
it can be shown that, generically, the snapshots can be decomposed using the Ritz vectors $z_j$ (normalized in the $\ell_2$ norm to $\|z_j\|_2=1$) as the spatial spanning set, and with the temporal coefficients provided by the Ritz values $\lambda_j = |\lambda_j| e^{\ii\omega_j \delta t}$:
 \begin{equation}\label{eq:f_i-reconstruct}
 \f_i =  \sum_{j=1}^m z_j \mathfrak{a}_j \lambda_j^{i-1} \equiv 
 \sum_{j=1}^m z_j \mathfrak{a}_j |\lambda_j|^{i-1} e^{\ii\omega_j (i-1)\delta t},\;\;i=1,\ldots , m .
 \end{equation}
In order to identify the most important coherent structures in the dynamic process, the representations \eqref{eq:f_i-reconstruct} are attempted with a smaller number of modes -- the task is to determine $\ell <m$, indices $\varsigma_1 < \cdots <\varsigma_{\ell}$ and the coefficients $\alpha_j$ to achieve high fidelity of the snapshot representations
\begin{equation}\label{eq:f_i-reconstruct-ell}
\f_i \approx  \sum_{j=1}^\ell z_{\varsigma_j} \alpha_j \lambda_{\varsigma_j}^{i-1} ,\;\;i=1,\ldots , m .
\end{equation}
A solution of this sparse reconstruction problem is proposed in \cite{jovschnicPOF14} as a combination of compressed sensing and convex optimization -- the resulting sparsity promoting DMD (DMDSP) has become an important addition to the DMD technology. One of the key numerical steps in DMDSP is solution of a structured least squares problem for the coefficients $\alpha_j$ in (\ref{eq:f_i-reconstruct-ell}) via the normal equations. 

From the numerical point of view, using normal equations for least squares solution is an ill-advised approach because the classical spectral condition number of the problem is squared. We examine this issue,  provide a floating point error analysis and analyze condition numbers that determine sensitivity of the problem. This allows us to explain good results obtained by DMDSP, and to determine the limits for applicability of the normal equations approach. 
To ensure robustness of the least squares representation (\ref{eq:f_i-reconstruct-ell}), we develop a QR factorization based algorithm. Together with detailed descriptions and numerical analysis of the new proposed algorithms, we provide guidelines for numerical software development.  

The material is organized as follows. In \S \ref{S=Preliminaries}, we briefly review the DMD and in \S \ref{zd:SS=Reconstr-selected} we state the problem of snapshot reconstruction using the selected modes; the sparsity promoting DMD (DMDSP) is reviewed in \S \ref{SS=DMDSP}. In \S \ref{S=Structure-sensitivity} we first formulate the reconstruction problem in more generality by allowing weighted reconstruction error -- each snapshot carries a weight that determines its importance in the overall error. Then, in \S \ref{SS=Normal_Eq_solution}, the explicit formulas for the normal equations solutions are given for this weighted error. In \S \ref{SS=DMDSP-improved} we propose a modification of the DMDSP -- its polishing phase can be considerably improved by replacing the variation of the Lagrangian technique with deployment of certain canonical projections. Numerical properties of the normal equations approach are studied in detail in \S \ref{S=Normal-Sensitivity}.  We identify the relevant condition number and argue that the underlying Khatri--Rao product structure of the normal equations usually ensures numerical accuracy despite potentially high condition numbers of both the Vandermonde matrix of Ritz values and of the Ritz vectors. To the best of our knowledge, this aspect of the DMDSP has not been considered before. 

However, normal equations have its limitations, and we illustrate this by contrived small dimensional examples. This motivates the development in \S \ref{SS=QR-LS-solver}, where we consider solution methods via the QR factorization -- the corrected seminormal solution in \S \ref{SS=Corrected-Semi-Normal} and the pure QR based procedure. The key technical difficulty is the computation of the QR factorization of the least squares matrix; this is resolved in \S \ref{SS=QRF(S)} where we propose an efficient recursive scheme that exploits the Khatri--Rao product structure involving a Vandermonde matrix. Fine details such as using only real arithmetic in the case of real snapshots are discussed in \S \ref{SSS=Real-data-intro}.
Finally, in \S \ref{S=Matlab-sources} we provide Matlab source codes for some of the presented algorithms. 

%% file: SOURCES2/Section_Sparse_Reconstruction.tex
\section{Preliminaries}\label{S=Preliminaries}
\noindent For the reader's convenience, and to introduce the notation, we briefly review the DMD, the snapshot reconstruction task (\ref{eq:f_i-reconstruct}, \ref{eq:f_i-reconstruct-ell}), and the DMDSP.
For a more extensive study of DMD and its variations and applications, see e.g. 
\cite{schmid2010}, \cite{Chen:2012jh}, 
\cite{2015arXiv150203854H}, \cite{Dawson2016}, 
\cite{Hemati:2014jm}.

\subsection{DMD and its variations}
\noindent The DMD \cite{schmid2010} is outlined in Algorithm \ref{zd:ALG:DMD}.
\begin{algorithm}[H]
	\caption{{$[Z_k, \Lambda_k]=\mathrm{DMD}(\X_m,\Y_m)$}}
	\label{zd:ALG:DMD}
	\begin{algorithmic}[1]
		\REQUIRE \  \\		
		\begin{itemize} 
			\item $\X_m=(\x_1,\ldots,\x_m), \Y_m=(\y_1,\ldots,\y_m)\in \mathbb{C}^{n\times m}$ that define a sequence of snapshots pairs $(\x_i,\y_i\equiv \A \x_i)$. (Tacit assumption is that $n$ is large and that $m \ll n$.)
		\end{itemize}
		\STATE $[U,\Sigma, \Phi]=svd(\X_m)$ ; \COMMENT{\emph{The thin SVD: $\X_m = U \Sigma \Phi^*$, $U\in\mathbb{C}^{n\times m}$, $\Sigma=\mathrm{diag}(\sigma_i)_{i=1}^m$}}
		\STATE Determine numerical rank $k$.
		\STATE Set $U_k=U(:,1:k)$, $\Phi_k=\Phi(:,1:k)$, $\Sigma_k=\Sigma(1:k,1:k)$ 	
		\STATE ${S}_k = (({U}_k^* \Y_m) \Phi_k)\Sigma_k^{-1}$; \COMMENT{\emph{Schmid's formula for the Rayleigh quotient $U_k^* \A U_k$}}
		\STATE $[B_k, \Lambda_k] = \mathrm{eig}(S_k)$ \COMMENT{$\Lambda_k=\mathrm{diag}(\lambda_j)_{j=1}^k$; $S_k B_k(:,j)=\lambda_j B_k(:,j)$; $\|B_k(:,j)\|_2=1$}\\
		\COMMENT{\emph{We always assume the generic case that $S_k$ is diagonalizable, i.e. that in Line 5. the function \texttt{eig()} computes the full column rank matrix $B_k$ of the eigenvectors of $S_k$. In general, $B_k$ may be ill-conditioned.}}
		\STATE $Z_k = U_k B_k$ \COMMENT{\emph{Ritz vectors}}
		\ENSURE $Z_k$, $\Lambda_k$
	\end{algorithmic}
\end{algorithm}

\subsubsection{DDMD RRR}\label{SS=DDMD-RRR}
Recently, in \cite{DDMD-SISC-2018}, we proposed  an enhancement of the Algorithm \ref{zd:ALG:DMD} -- Refined Rayleigh Ritz Data Driven Modal Decomposition. It follows the DMD scheme, but it further allows data driven refinement of the Ritz vectors and it provides data driven computable residuals. For the reader's convenience, the method is outlined in Algorithm \ref{zd:ALG:DMD:RRR}; for detailed description and analysis, and a Matlab implementation we refer to \cite{DDMD-SISC-2018}.

\begin{algorithm}[h]
	\caption{$[Z_k, \Lambda_k, r_k, \rho_k]=\mathrm{DDMD\_RRR}(\X_m,\Y_m; \epsilon)$ \{\emph{Refined Rayleigh-Ritz DDMD}\}}
	\label{zd:ALG:DMD:RRR}
	\begin{algorithmic}[1]
		\REQUIRE \  \\		
		\begin{itemize} 
			\item $\X_m=(\x_1,\ldots,\x_m), \Y_m=(\y_1,\ldots,\y_m)\in \mathbb{C}^{n\times m}$ that define a sequence of snapshots pairs $(\x_i,\y_i\equiv \A(\x_i))$. (Tacit assumption is that $n$ is large and that $m \ll n$.)
			\item Tolerance level $\epsilon$ for numerical rank determination.		
		\end{itemize}
		\STATE $D_x=\mathrm{diag}(\|\X_m(:,i)\|_2)_{i=1}^m$; $\X_m^{(1)}= \X_m D_x^{\dagger}$; $\Y_m^{(1)} =\Y_m D_x^{\dagger}$\COMMENT{\emph{Optional scaling.}}
		\STATE $[U,\Sigma, V]=svd(\X_m^{(1)})$ ; \COMMENT{\emph{The thin SVD: $\X_m^{(1)} = U \Sigma V^*$, $U\in\mathbb{C}^{n\times m}$, $\Sigma=\mathrm{diag}(\sigma_i)_{i=1}^m$}}
		\STATE Determine numerical rank $k$, with the threshold $\epsilon$: $k = \max\{ i \; : \; \sigma_i \geq \sigma_1 \epsilon \}$.
		\STATE Set $U_k=U(:,1:k)$, $V_k=V(:,1:k)$, $\Sigma_k=\Sigma(1:k,1:k)$ 	
		\STATE ${B}_k = \Y_m^{(1)} (V_k\Sigma_k^{-1})$; \COMMENT{\emph{Schmid's data driven formula for $\A U_k$}}
		\STATE $[Q, R]=qr(\begin{pmatrix} U_k, & B_k\end{pmatrix})$; \COMMENT{\emph{The thin QR factorization: $\begin{pmatrix} U_k, & B_k\end{pmatrix} = QR$; $Q$ not computed}}
		\STATE $S_k = \mathrm{diag}(\overline{R_{ii}})_{i=1}^k R(1:k,k+1:2k)$ \COMMENT{\emph{$S_k = U_k^* \A U_k$ is the Rayleigh quotient}}
		\STATE $\Lambda_k = \mathrm{eig}(S_k)$ \COMMENT{$\Lambda_k=\mathrm{diag}(\lambda_i)_{i=1}^k$; \emph{Ritz values, i.e. eigenvalues of $S_k$}}
		\FOR{$i=1,\ldots, k$}
		\STATE $R_{\lambda_i} = \left( \begin{smallmatrix} R(1:k,k+1:2k)-\lambda_i R(1:k,1:k)\cr R(k+1:2k,k+1:2k) \end{smallmatrix}\right)$
		\STATE $[\sigma_{\lambda_i},w_{\lambda_i} ]=svd_{\min}(R_{\lambda_i})$;
				 \COMMENT{\emph{Minimal singular value of $R_{\lambda_i}$ and the corresponding right singular vector}}
		\STATE $W_k(:,i)=w_{\lambda_i}$ ; 
		$r_k(i) = \sigma_{\lambda_i}$ \COMMENT{\emph{Optimal residual, $\sigma_{\lambda_i}=\|R_{\lambda_i} w_{\lambda_i}\|_2$}}
		\STATE $\rho_k(i)=w_{\lambda_i}^* S_k w_{\lambda_i}$ \COMMENT{\emph{Rayleigh quotient, $\rho_k(i)= (U_k w_{\lambda_i})^* \A (U_k w_{\lambda_i})$}}
		\ENDFOR
		\STATE $Z_k = U_k W_k$ \COMMENT{\emph{Refined Ritz vectors}}
		\ENSURE $Z_k$, $\Lambda_k$, $r_k$, $\rho_k$
	\end{algorithmic}
\end{algorithm}

\subsection{Sparse reconstruction using selected Ritz pairs}\label{zd:SS=Reconstr-selected}
\noindent
In matrix form, (\ref{eq:f_i-reconstruct}) is compactly written as
\begin{equation}\label{eq:reconstruction-formula}
\X_m = \begin{pmatrix} z_1 & z_2 & \ldots & z_m \end{pmatrix} \begin{pmatrix} 
{\mathfrak{a}}_1 &  &  & \cr
& {\mathfrak{a}}_2 &  &   \cr 
&  & \ddots &        \cr 
&        &  & {\mathfrak{a}}_m\end{pmatrix}
\begin{pmatrix} 
1 & \lambda_1 & \ldots & \lambda_1^{m-1} \cr
1 & \lambda_2 & \ldots & \lambda_2^{m-1} \cr
\vdots & \vdots & \ldots & \vdots \cr
1 & \lambda_m & \ldots & \lambda_m^{m-1} \cr
\end{pmatrix} ,
\end{equation}
where the coefficients $\mathfrak{a}_j$ are computed as
\begin{equation}\label{zd:eq:DMD-aplitudes}
({\mathfrak{a}}_j)_{j=1}^m = Z_m^\dagger\X_m(:,1).
\end{equation}
In the framework of the Schmid's DMD, the amplitudes are usually determined by the formula (\ref{zd:eq:DMD-aplitudes}). Since $Z_m = U_m B_m$ (see line 6. in Algorithm \ref{zd:ALG:DMD}) and $U_m^* U_m=\Id_m$, instead of applying the pseudoinverse of the explicitly computed $Z_m$, one would use the more efficient formula
\begin{equation}
Z_m^\dagger \X_m(:,1) = B_m^{-1}(U_m^* \X_m(:,1)).
\end{equation}
For a formal proof of the existence of the representation (\ref{eq:reconstruction-formula}) and relation to the formulation  via the Krylov decomposition and the Frobenius companion matrix, we refer to \cite{AIMdyn-Vand-Cauchy-2018}.


The main goal of the representations (\ref{eq:f_i-reconstruct}), (\ref{eq:reconstruction-formula}) is revealing the  underlying structure in the snapshots.
It is best achieved with an approximate representation using as few as possible Ritz pairs (i.e. approximate eigenpairs). Ideally, the reconstruction is successful with small number of Ritz vectors that have small residuals (i.e. corresponding to selected eigenvectors of the underlying operator that is accessible only through the sequence of snapshots $\f_i$). 

If we desire to take only $\ell<m$ most relevant Ritz pairs, then the question is how to determine $\ell$ and what indices $\varsigma_1, \ldots, \varsigma_\ell$ should be selected to achieve good approximation
\begin{equation}\label{eq:reconstruction-formula-ell}
\X_m \approx \begin{pmatrix} z_{\varsigma_1} & z_{\varsigma_2} & \ldots & z_{\varsigma_{\ell}} \end{pmatrix} \begin{pmatrix} 
{\mathfrak{a}}_{\varsigma_1} &  &  & \cr
& {\mathfrak{a}}_{\varsigma_2} &  &   \cr 
&  & \ddots &        \cr 
&        &  & {\mathfrak{a}}_{\varsigma_\ell}\end{pmatrix}
\begin{pmatrix} 
1 & \lambda_{\varsigma_1} & \ldots & \lambda_{\varsigma_1}^{m-1} \cr
1 & \lambda_{\varsigma_2} & \ldots & \lambda_{\varsigma_2}^{m-1} \cr
\vdots & \vdots & \ldots & \vdots \cr
1 & \lambda_{\varsigma_\ell} & \ldots & \lambda_{\varsigma_\ell}^{m-1} \cr
\end{pmatrix} \equiv Z_{\varsigma} D_{{\mathfrak{a}}} \Vanderm_{\varsigma} .
\end{equation}

This seems difficult task for practical computation. In general, here we assume the availability of a \emph{reconstruction wizard} that selects $z_{\varsigma_1},\ldots, z_{\varsigma_{\ell}}$ so that in (\ref{eq:reconstruction-formula-ell}) the reconstruction error $\| \X_m - Z_{\varsigma} D_{{\mathfrak{a}}} \Vanderm_{\varsigma} \|_F^2$ is as small as possible. An example of such a wizard is an optimizer with (relaxed) sparsity constraints, e.g the ADMM (Alternating Directions Method of Multipliers) which is used in the sparsity promoting DMD \cite{jovschnicPOF14}, that we briefly review in \S \ref{SS=DMDSP}. Another strategy, proposed in \cite[\S 3]{KOU2017109}, is to choose modes that are dominant with respect to their influence over all snapshots. 

For any strategy, once the modes are  selected, instead of using the coefficients $\mathfrak{a}_{\varsigma_i}$ from the full reconstruction, one can achieve higher fidelity of (\ref{eq:reconstruction-formula-ell}) by recomputing them by solving a least squares problem with fixed $z_{\varsigma_1},\ldots, z_{\varsigma_{\ell}}$. Here, optionally, we can use data driven  refinements of the selected  Ritz pairs $(z_{\varsigma_j},\lambda_{\varsigma_j})$, see \cite{DDMD-SISC-2018}.

\subsubsection{DMDSP: sparsity promoting DMD}\label{SS=DMDSP}

One way to set up a computational procedure is to formulate it as least squares approximation with sparsity constraints
\begin{equation}\label{eq:LS-sparse-constraint}
\| \X_m - Z_k D_{{\mathfrak{a}}} \Vanderm_{k,m} \|_F^2 + \gamma \|{\mathfrak{a}}\|_0 \longrightarrow \min_{\widetilde{\mathfrak{a}}} ,
\end{equation}
where $\|{\mathfrak{a}}\|_0$ denotes the number of nonzero entries in the vector ${\mathfrak{a}}=(\mathfrak{a}_i)_{i=1}^k$ of the new coefficients; the parameter $\gamma\geq 0$  penalizes non--sparsity of the solution ${\mathfrak{a}}$. The measure of sparsity $\|{\mathfrak{a}}\|_0$ is in practical computations relaxed by using the $\ell_1$ norm, $\|{\mathfrak{a}} \|_1$, which turns (\ref{eq:LS-sparse-constraint}) into a convex optimization problem
\begin{equation}\label{eq:LS-sparse-constraint-relaxed}
\| \X_m - Z_k D_{{\mathfrak{a}}} \Vanderm_{k,m} \|_F^2 + \gamma \|{\mathfrak{a}}\|_1 \longrightarrow \min_{{\mathfrak{a}}}.
\end{equation}
In \cite{jovschnicPOF14}, for a given value of $\gamma$, the problem is solved in two steps: 
\begin{enumerate}
	\item (\ref{eq:LS-sparse-constraint-relaxed}) is solved using the ADMM, and the optimal ${\mathfrak{a}}$ is sparsified by setting to zero its entries that are in modulus below a given threshold. Let $j_1,\ldots, j_s$ be the indices of thus annihilated entries -- they define sparsity pattern. Define $E$ to be the matrix whose columns are the columns of the identity with the indices  $j_1,\ldots, j_s$. Then the obtained sparsity pattern of ${\mathfrak{a}}$ can be characterized by $E^T {\mathfrak{a}}=\0$. 
\item Once (\ref{eq:LS-sparse-constraint-relaxed}) has identified the sparsity structure, the coefficients $\mathfrak{a}_i$ are computed by solving the least squares problem with sparsity constraint:
\begin{equation}\label{eq:DMDSP:problem}
\| \X_m - Z_k D_{{\mathfrak{a}}} \Vanderm_{k,m} \|_F^2 \longrightarrow \min_{{\displaystyle {\mathfrak{a}}, E^T{\mathfrak{a}}=\0}}
\end{equation}
In DMDSP terminology, this is the \emph{polishing}  part of the computation.
\end{enumerate}
Since an optimal value of $\gamma$ may be problem dependent, the above two-step procedure is repeated over a grid of (dozens or hundreds) values of $\gamma$.

An advantage of DMDSP is that it is a black-box procedure; the wizard is simply a convex optimization solver. However, it requires suitable range for the parameter $\gamma$, which, to our best understanding, is determined experimentally for each particular problem.\footnote{See e.g. the software \cite{dmdsp-matlab} for test examples in \cite{jovschnicPOF14}.} Further, ADMM uses the augmented Lagrangian that requires an additional penalty parameter $\rho$, which means that the user must supply two parameters (see \cite[\S A.]{jovschnicPOF14}).  

The optimization problem (\ref{eq:DMDSP:problem}) is solved by variation of the Lagrangian (see \cite[Appendix C]{jovschnicPOF14}). This can be done more efficiently and we discuss it in \S \ref{SS=DMDSP-improved}. Further, in the case of real data, the approximant needs to be real as well; this is naturally achieved if the selected modes appear in complex conjugate pairs. It is not clear whether the optimizer in DMDSP is so tuned to respect this structure. The issue of computations with real data is discussed in detail in \S \ref{SSS=Real-data-intro}.

We complete this review with an important observation.
\begin{remark}\label{REM:DMDSP-Compr-Sens}
{\em
The minimizations (\ref{eq:LS-sparse-constraint}), (\ref{eq:LS-sparse-constraint-relaxed}), with the sparsity constraint, look analogous to the compressed sensing problems; in fact \cite{jovschnicPOF14} motivated the development of DMDSP as a combination of compressed sensing and convex optimization techniques. It should be noted, however, that the snapshot reconstruction problem (\ref{eq:reconstruction-formula-ell}) is heavily overdetermined and generically with unique solution, while in the compressed sensing framework one has an underestimated least squares problem and the sparsity constraint is imposed over a manifold of solutions. These are two fundamentally different situations. We further comment this issue in \S \ref{SSS=Pivoting-important}.
}	
\end{remark}

%% file: SOURCES2/Section_Reconstruction_Algorithms.tex
\section{Snapshot reconstruction: structure, sensitivity and condition numbers}\label{S=Structure-sensitivity}
\noindent Suppose we have selected $\ell$ modes, re-indexed so that they are the leading ones, $1,\ldots,\ell$, and that we seek an approximate snapshot reconstruction
\begin{equation}\label{eq:reconstruct-ell}
\f_i \approx \sum_{j=1}^\ell z_j \alpha_j \lambda_j^{i-1},\;\;i=1,\ldots, m.
\end{equation}
 The pairs $(\lambda_j,z_j)$ are approximate eigenpairs of $\A$, e.g. the Ritz pairs. In terms of Algorithm \ref{zd:ALG:DMD} or Algorithm \ref{zd:ALG:DMD:RRR}, $Z_{\ell}\equiv \begin{pmatrix} z_1 & \ldots & z_{\ell}\end{pmatrix}=U_k \widetilde{B}_\ell$ with some $k\times \ell$ matrix $\widetilde{B}_\ell$, and $\ell\leq k\leq \min(m,n)$. 
 
 
 In general, selection of modes will be determined by a wizard, aimed at achieving certain goals; an example is the DMDSP, reviewed in \S \ref{SS=DMDSP}. 
 But once the $(\lambda_j,z_j)$'s have been selected, we may ask whether the coefficients $\alpha_j={\mathfrak{a}}_j$ can be recomputed so that the reconstruction (\ref{eq:reconstruct-ell}) is improved over all snapshots. This is e.g. the polishing phase (\ref{eq:DMDSP:problem}) in DMDSP.
%

\subsection{Setting the scene}
\noindent The task is to determine, for given $(\lambda_j,z_j)$'s and nonnegative weights $\ww_i$,  the $\alpha_j$'s to achieve
\begin{equation}\label{eq:rec-error-min}
\sum_{i=1}^m \ww_i^2 \| \f_i - \sum_{j=1}^\ell z_j \alpha_j \lambda_j^{i-1}\|_2^2 \longrightarrow \min .
\end{equation}
Set $\WW=\diag(\ww_i)_{i=1}^m$.  The weights $\ww_i>0$ are used to emphasize snapshots whose reconstruction is more important; with suitable choices of $\WW$ we can target discrete time subintervals of particular interest, or e.g. introduce forgetting factors that give less importance to older information. We believe that this added functionality will  prove useful in applications and the non-autonomous setting. 
%

Here we use the $\|\cdot\|_2$ norm as the most common choice; if one wants to distinguish the importance of different observables (that might be of different physical nature, expressed in different units, on different scales and with  different levels of uncertainty), then $\|\cdot \|_2$ can be replaced with an energy norm $\| \sqrt{\Omega} \cdot\|_2$, where $\Omega$ is positive definite. (If $\Omega$ is not diagonal, then $\sqrt{\Omega}$ stands for the upper triangular Cholesky factor.) This important modification adds one new level of technical details, and it is omitted in this paper. For a detailed analysis of DMD in the $\Omega$--induced inner product we refer the reader to \cite{DDMD-SISC-2018}.

To introduce all necessary notation, define $\bfLambda=\mathrm{diag}(\lambda_j)_{j=1}^{\ell}$,
$$
\vec{\bfalpha} = \begin{pmatrix} \alpha_1 \cr \alpha_2 \cr \cdot\cr \alpha_{\ell}\end{pmatrix},\;\;
\Delta_{\bfalpha} = \begin{pmatrix} \alpha_1 & 0 & \cdot & 0 \cr 
0 & \alpha_2 & \cdot & \cdot \cr
\cdot & \cdot & \cdot & 0 \cr
0 & \cdot & 0 & \alpha_{\ell}\end{pmatrix},\;\;
\Lambda_i = \begin{pmatrix} \lambda_1^{i-1}\cr \lambda_2^{i-1} \cr \cdot \cr \lambda_{\ell}^{i-1} \end{pmatrix},\;\;
\Delta_{\Lambda_i} = \begin{pmatrix} \lambda_1^{i-1} & 0 & \cdot & 0 \cr 
0 & \lambda_2^{i-1} & \cdot & \cdot \cr
\cdot & \cdot & \cdot & 0 \cr
0 & \cdot & 0 & \lambda_{\ell}^{i-1}\end{pmatrix}\equiv\bfLambda^{i-1} ,
$$
and write the objective (\ref{eq:rec-error-min}) as the function of $\vec{\bfalpha}$,
\begin{equation}\label{eq:Omega(alpha)}
\Omega^2(\vec{\bfalpha}) \equiv \| [ \X_m - Z_{\ell} \Delta_{\bfalpha} \begin{pmatrix} \Lambda_1 & \Lambda_2 & \ldots & \Lambda_{m}\end{pmatrix}]\WW \|_F^2 \longrightarrow\min,
\end{equation}
where
\begin{equation}\label{eq:Vanderm_ell_m}
\begin{pmatrix} \Lambda_1 & \Lambda_2 & \ldots & \Lambda_{m}\end{pmatrix} = 
\begin{pmatrix} 
1 & \lambda_1 & \ldots & \lambda_1^{m-1} \cr
1 & \lambda_2 & \ldots & \lambda_2^{m-1} \cr
\vdots & \vdots & \ldots & \vdots \cr
1 & \lambda_\ell & \ldots & \lambda_\ell^{m-1} \cr
\end{pmatrix} \equiv \Vanderm_{\ell,m}\in\C^{\ell\times m} .
\end{equation}
Schematically, we have, assuming $n > m$, 
\begin{displaymath}
\underbrace{\begin{pmatrix}
	x & x & x & x & x \cr
	x & x & x & x & x \cr
	x & x & x & x & x \cr
	x & x & x & x & x \cr
	x & x & x & x & x \cr
	x & x & x & x & x \cr
	x & x & x & x & x \cr
	x & x & x & x & x \cr
	\end{pmatrix}}_{\X_m} \approx
\underbrace{\begin{pmatrix}
	\bullet & \bullet \cr
	\bullet & \bullet \cr
	\bullet & \bullet \cr
	\bullet & \bullet \cr
	\bullet & \bullet \cr
	\bullet & \bullet \cr
	\bullet & \bullet \cr
	\bullet & \bullet \cr
	\end{pmatrix}}_{Z_{\ell}}
\overbrace{\begin{pmatrix} \star & 0 \cr 0 & \star\end{pmatrix}}^{\Delta_{\bfalpha}}
\underbrace{\begin{pmatrix}
	+ & + & + & + & +  \cr
	+ & + & + & + & +  
	\end{pmatrix}}_{\begin{pmatrix} \Lambda_1 & \Lambda_2 & \ldots & \Lambda_{m}\end{pmatrix}} .
\end{displaymath}
%
Since $\X_m=(\f_1,\ldots, \f_m)$, we can rewrite the above expression as follows:
\begin{eqnarray}
	\Omega^2(\vec{\bfalpha}) &=& 	\| \left[\begin{pmatrix} \f_1 & \f_2 & \ldots & \f_m\end{pmatrix} - Z_{\ell} \Delta_{\bfalpha} \begin{pmatrix} \Lambda_1 & \Lambda_2 & \ldots & \Lambda_{m}\end{pmatrix}\right] \WW\|_F^2 \nonumber \\ &=&  \| (\WW \otimes \Id_n)\left[\begin{pmatrix}  \f_1 \cr \vdots \cr  \f_m\end{pmatrix} - \begin{pmatrix}  Z_{\ell}\Delta_{\bfalpha}\Lambda_1 \cr \vdots \cr  Z_{\ell}\Delta_{\bfalpha}\Lambda_m\end{pmatrix}\right]\|_2^2 
	= \| \begin{pmatrix} \ww_1 \f_1 \cr \vdots \cr \ww_m \f_m\end{pmatrix} -  \begin{pmatrix} \ww_1 Z_{\ell}\Delta_{\Lambda_1} \cr \vdots \cr \ww_m Z_{\ell}\Delta_{\Lambda_m}\end{pmatrix} \vec{\bfalpha}\|_2^2 . \label{eq:Omega:f-ZDa}
\end{eqnarray}
%
In principle, (\ref{eq:Omega:f-ZDa}) is a standard linear least squares problem that can be solved using an off-the-shelf solver for the generic "$\|Ax-b\|_2\rightarrow\min$" problem. This is certainly not optimal (here we assume $n\gg m > \ell$) because it ignores the particular structure of the  coefficient matrix, which is of dimensions $mn \times \ell$, not sparse, and the direct solver complexity is $O(mn\ell^2)$ flops. Further, to understand the numerical accuracy of the computed approximations, and to identify relevant condition numbers, we need to take into account the structure that involves the Ritz pars $(z_j,\lambda_j)$. These issues are important and in this section we provide the details.
 
The first step is to remove the ambient space dimension $n$ and to replace it with $\ell$. To that end, let $Z_{\ell}=Q\left(\begin{smallmatrix} R \cr 0\end{smallmatrix}\right)$ be the QR factorization. Note that $R$ is of full rank.

Then, using the unitary invariance of the Euclidean norm, the above can be written as
\begin{equation}\label{eq:Omega:g-RDa}
\Omega^2(\vec{\bfalpha}) = \| \begin{pmatrix} \ww_1 Q^* \f_1 \cr \vdots \cr \ww_m Q^*\f_m\end{pmatrix} - \begin{pmatrix} 
\ww_1 \left(\begin{smallmatrix} R \cr 0\end{smallmatrix}\right) \Delta_{\Lambda_1} \cr \vdots \cr \ww_m \left(\begin{smallmatrix} R \cr 0\end{smallmatrix}\right) \Delta_{\Lambda_m}\end{pmatrix} \vec{\bfalpha}\|_2^2 .
\end{equation}
Now partition each $Q^*\f_i$ as $Q^*\f_i = \left( \begin{smallmatrix} \g_i \cr \h_i\end{smallmatrix}\right)$ where $\g_i$ is $\ell\times 1$ and $\h_i$ s $(n-\ell)\times 1$.
The objective function becomes
\begin{eqnarray}\label{eq:Omega(alpha)}
\Omega^2(\vec{\bfalpha}) &=& \| (\WW\otimes \Id_{\ell}) \left[\begin{pmatrix} \g_1 \cr \vdots \cr \g_m\end{pmatrix} - \begin{pmatrix} 
R \Delta_{\Lambda_1} \cr \vdots \cr  R  \Delta_{\Lambda_m}\end{pmatrix} \vec{\bfalpha} \right]\|_2^2 + \sum_{i=1}^m \ww_i^2 \| \h_i\|_2^2 \\
&\equiv&  \|(\WW\otimes \Id_{\ell}) \left[\begin{pmatrix} \g_1 \cr \vdots \cr \g_m\end{pmatrix} -
(\Id_m\otimes R) \begin{pmatrix} 
\Delta_{\Lambda_1} \cr \vdots \cr  \Delta_{\Lambda_m}\end{pmatrix} \vec{\bfalpha}\right]\|_2^2 + \sum_{i=1}^m \ww_i^2 \| \h_i\|_2^2 \\
&\equiv&  \|\begin{pmatrix} \ww_1 \g_1 \cr \vdots \cr \ww_m \g_m\end{pmatrix} -
(\WW\otimes R) \begin{pmatrix} 
\Delta_{\Lambda_1} \cr \vdots \cr  \Delta_{\Lambda_m}\end{pmatrix} \vec{\bfalpha}\|_2^2 + \sum_{i=1}^m \ww_i^2 \| \h_i\|_2^2 ,
\end{eqnarray}
where $\otimes$ denotes the Kronecker product.
Note that $\h_i$ is the component of the corresponding $\f_i$ that is orthogonal to the range of $Z_{\ell}$, thus beyond the reach of optimization with respect to $\vec{\bfalpha}$. Hence, we actually need the economy-size QR factorization $Z_{\ell}=Q(:,1:\ell) R$; $\g_i=Q(:,1:\ell)^*\f_i$. 

\begin{remark}\label{REM:POD-reconstruct}
	{\em Since in a DMD framework  $Z_{\ell}=U_k \widetilde{B}_{\ell}$, $\widetilde{B}_{\ell}\in\C^{k\times\ell}$, the factorization of $Z_{\ell}$ is obtained from the QR factorization $\widetilde{B}_{\ell}=Q' R$, where by the essential uniqueness $Q(:,1:\ell)=U_k Q'$.
(Even if the column of $Z_\ell$ are the refined Ritz vectors, they have a representation of the form $Z_{\ell}=U_k \widetilde{B}_{\ell}$, only with a different matrix $\widetilde{B}_{\ell}$.) Hence, the QR factorization of $Z_{\ell}$ is not necessarily computed explicitly from the explicitly computed $Z_{\ell}$; the more economic way is using the QR factorization of $\widetilde{B}_{\ell}$. 

Further, if $\X_m = U_m \Sigma_m V_m^*$ is the thin SVD of $\X_m$ then
\begin{eqnarray*}
\Omega^2(\vec{\bfalpha}) &=& 	\| \left[ U_m\Sigma_m V_m^* - U_k \widetilde{B}_{\ell} \Delta_{\bfalpha} \begin{pmatrix} \Lambda_1 & \Lambda_2 & \ldots & \Lambda_{m}\end{pmatrix} \right] \WW \|_F^2 \;\;\;\mbox{(since $U_k=U_m U_m^* U_k$)}\\
&=& \| \left[ \Sigma_m V_m^* - \begin{pmatrix} \widetilde{B}_{\ell} \cr \0_{m-k,\ell}\end{pmatrix}\Delta_{\bfalpha} \begin{pmatrix} \Lambda_1 & \Lambda_2 & \ldots & \Lambda_{m}\end{pmatrix} \right] \WW \|_F^2 \\ &=& 
\| \left[ \Sigma_k V_k^* - \widetilde{B}_{\ell} \Delta_{\bfalpha} \begin{pmatrix} \Lambda_1 & \Lambda_2 & \ldots & \Lambda_{m}\end{pmatrix}\right]\WW \|_F^2 + \sum_{j=k+1}^m \sigma_j^2 \|V_m^*(j,:)\WW\|_2^2 ,
\end{eqnarray*}
and (as in the unweighted case considered in \cite{jovschnicPOF14}) the reconstruction is done  with the projections of the snapshots onto the span of the leading $k$ left singular vectors (principal components). In the corresponding basis, the snapshots are the columns of $\Sigma_k V_k^*$ ($\f_i\equiv (\Sigma_k V_k^*)(:,i)$), the Ritz vectors are in $\widetilde{B}_{\ell}$ ($Z_\ell\equiv \widetilde{B}_{\ell}$), and we may proceed with the QR factorization $\widetilde{B}_{\ell}=Q' R$, thus removing the large dimension $n$ in the very first step. This is included as a special case in the generic description (\ref{eq:Omega:f-ZDa} -- \ref{eq:Omega(alpha)}) which appears in other applications (outside the DMD framework; see e.g. \cite{lev-ari2005}) and it is thus preferred as a general form of the structured LS problem. 
}
\end{remark}

\begin{remark}
{\em
An efficient alternative approach to reducing the dimension is the QR--compressed scheme \cite[\S 5.2.1]{DDMD-SISC-2018}, which replaces the ambient space dimension $n$ with $m+1$ even before the DMD computation. 	
In this projected framework, the QR factorization of $Z_{\ell}$ then reduces the dimension from $m+1$ to $\ell$. This has been successfully used in \cite{AIMdyn-Vand-Cauchy-2018}, where $Z_{\ell}$ is computed directly (without the DMD) from the eigenvectors of the companion matrix (inverse of the Vandermonde matrix from (\ref{eq:reconstruction-formula})).
}	
\end{remark}
\subsection{Structure of the LS problem and normal equations solution}
\label{SS=Normal_Eq_solution}
\noindent The objective function $\Omega(\vec{\bfalpha})$ has a very particular structure that allows a rather elegant explicit formula \cite[\S II.B]{jovschnicPOF14} for the optimal $\vec{\bfalpha}$ via the normal equations in the unweighted case ($\WW=\Id_m$). Here,  we generalize the formula to the weighted case. Then, we discuss an additional structure. 

\begin{theorem}\label{TM:weighted_LS}
With the notation as above, the unique solution $\vec{\bfalpha}$ of the LS problem (\ref{eq:rec-error-min}) is 
	\begin{equation}\label{eq:alpha:normal:eq}
	\vec{\bfalpha} = [ (R^* R) \circ (\overline{\Vanderm_{\ell,m}\WW^2 \Vanderm_{\ell,m}^*}) ]^{-1} [ (\overline{\Vanderm_{\ell,m}\WW}\circ (R^* G\WW))\eb ] ,
	\end{equation}
	where $G=\begin{pmatrix} \g_1 & \ldots & \g_m\end{pmatrix}$, $\eb=\begin{pmatrix} 1 & \ldots & 1 \end{pmatrix}^T$. In terms of the original data,
	\begin{equation}\label{eq:alpha:normal:eq:orig-data}
	\vec{\bfalpha} = [ (Z_{\ell}^*Z_{\ell}) \circ (\overline{\Vanderm_{\ell,m}\WW^2 \Vanderm_{\ell,m}^*}) ]^{-1} [ (\overline{\Vanderm_{\ell,m}\WW}\circ (Z_{\ell}^* \X_m\WW))\eb ] .
	\end{equation}
\end{theorem}
\begin{proof}
Note that $\vec{\bfalpha}$ actually solves 
\begin{equation}\label{zd:eq:g-S*alpha}
\| (\WW\otimes \Id_{\ell}) \left[ \vec{\g} -
S \vec{\bfalpha} \right]\|_2\longrightarrow\min,\;\;\mbox{where}\;\;\vec{\g}= \begin{pmatrix} \g_1 \cr \vdots \cr \g_m\end{pmatrix},\;\; S = (\Id_m\otimes R) \begin{pmatrix} 
 \Delta_{\Lambda_1} \cr \vdots \cr  \Delta_{\Lambda_m}\end{pmatrix} \equiv 
\begin{pmatrix} 
R \Delta_{\Lambda_1} \cr \vdots \cr  R \Delta_{\Lambda_m}\end{pmatrix} .
\end{equation}
Hence, the optimal  $\vec{\bfalpha}$ is given by the Moore--Penrose generalized inverse,  $\vec{\bfalpha}=S_{\ww}^\dagger \vec{\g}_{\ww}$, where
$\vec{\g}_{\ww}= (\WW\otimes \Id_{\ell})\vec{\g}$, $S_{\ww}=(\WW\otimes \Id_{\ell})S$, $S_{\ww}^\dagger = (S_{\ww}^* S_{\ww})^{-1}S_{\ww}^*$, i.e.
\begin{equation}\label{eq:S-dagger}
\vec{\bfalpha}=S_{\ww}^\dagger\vec{\g}_{\ww} = (S_{\ww}^* S_{\ww})^{-1}S_{\ww}^* \vec{\g}_{\ww}= (\sum_{k=1}^m \ww_k^2 \Delta_{\Lambda_k}^* R^* R \Delta_{\Lambda_k})^{-1} \sum_{i=1}^m \ww_i \Delta_{\Lambda_i}^* (\ww_i R^{*}\g_i) .
\end{equation}
Further, using the Hadamard matrix product $\circ$ and the element-wise conjugation $\overline{\;\cdot\;}$, we can write  
\begin{eqnarray}\label{eq:S*S}
\sum_{k=1}^m \ww_k^2 \Delta_{\Lambda_k}^* R^* R \Delta_{\Lambda_k} &=& \sum_{k=1}^m \ww_k^2 (R^* R)\circ (\overline{\Lambda_k}\Lambda_k^T) = (R^* R) \circ \sum_{k=1}^m \ww_k^2
\overline{\Lambda_k}\Lambda_k^T \\ 
&=& (R^* R) \circ (\overline{\Vanderm}_{\ell,m}\WW^2 \Vanderm_{\ell,m}^T)
= (R^* R) \circ (\overline{\Vanderm_{\ell,m} \WW^2 \Vanderm_{\ell,m}^*}) ,
\end{eqnarray}
\begin{equation}\label{eq:G}
\sum_{i=1}^m \ww_i \Delta_{\Lambda_i}^* (\ww_i R^{*}\g_i) = (\overline{\Vanderm_{\ell,m}\WW}\circ (R^* G \WW))\eb .
\end{equation}
{Note that Theorem \ref{TM:weighted_LS} does not use the Vandermonde structure of $\Vanderm_{\ell,m}$; it only requires full row rank; similarly the triangular structure of $R$ is not important, and we only require $R$ to have full column rank. Hence, its result can be stated in a more general form.}
\end{proof}
The formula (\ref{eq:alpha:normal:eq}, \ref{eq:alpha:normal:eq:orig-data}), which appears to be new, contains the formula from \cite[\S II.B]{jovschnicPOF14} as a special unweighted case. 
Since the Hadamard product of two positive definite matrices remains positive definite, the solution of (\ref{eq:alpha:normal:eq}) is obtained using the Cholesky factorization of 
$(R^* R) \circ (\overline{\Vanderm_{\ell,m} \WW^2 \Vanderm_{\ell,m}^*})$ followed by forward and backward substitutions. The complexity of (\ref{eq:alpha:normal:eq}) in terms of arithmetic operations is dominated by $O(m\ell^2) + O(\ell^3)$. Typically, $m\gg \ell$.

In fact, the relations (\ref{zd:eq:g-S*alpha} -- \ref{eq:G}), which implied (\ref{eq:alpha:normal:eq}), hide an elegant structure that we discuss in \S \ref{SSS=Khatri-Rao-intro}.

\subsubsection{On Khatri--Rao structure of the snapshot reconstruction problem}\label{SSS=Khatri-Rao-intro}
Recall, for two column partitioned matrices $A=(a_1,a_2,\ldots, a_n)\in\C^{p\times n}$ and $B=(b_1,b_2,\ldots, b_n)\in\C^{q\times n}$, their Khatri--Rao product is defined as 
\begin{equation}
A \odot B = \begin{pmatrix} a_1\otimes b_1 & a_2\otimes b_2 & \ldots & a_n\otimes b_n \end{pmatrix} \in \C^{pq\times n} .
\end{equation}
The following proposition contains well known facts; it is included solely for the reader's convenience. For more detailed study of matrix products involved in this section we refer to \cite[Chapters 4 and 5]{hor-joh-91}. 
\begin{proposition}\label{PR:KR:properties}
	The Khatri--Rao product satisfies the following relations: 
	\begin{itemize}
		\item For any matrices $A$, $B$, $C$, $D$ of appropriate dimensions
		\begin{equation}
		(AB) \odot (CD) = (A\otimes C) (B\odot D) .
		\end{equation}
		\item For any three matrices $A$, $B$, $C$ of appropriate dimensions and with $B$ diagonal,
		\begin{equation}\label{eq:KR:ABC}
		\mathrm{vec}(ABC) = (C^T \odot A) \diag(B) .
		\end{equation}
		(Recall, if $B$ is not diagonal, $\mathrm{vec}(ABC)=(C^T\otimes A)\mathrm{vec}(B)$.)
		\item For any two matrices $A$ and $B$ with the same number of columns,
		\begin{eqnarray}
		(A\odot B)^T (A\odot B) &=& (B\odot A)^T (B\odot A) = (A^T A)  \circ (B^T B) ,\\
		(A\odot B)^* (A\odot B) &=& (B\odot A)^* (B\odot A) = (A^* A) \circ (B^* B).\;\;
		\end{eqnarray}
	\end{itemize}
\end{proposition}

Let $\Pi$ be permutation matrix whose columns are the columns of the $(m\ell\times m\ell)$ identity taken in the order of the following permutation $\varpi_{\ell,m}$:
\begin{equation}\label{eq:varpi-ell,m}
\varpi_{\ell,m}= \left( \begin{smallmatrix}
1 &    2   &   3     & \ldots &  m          & m+1 & m + 2  &  m + 3  & \ldots & 2m          & \ldots  & m(\ell-1)+1        & \ldots         &   m\ell  \cr
1 & \ell+1 & 2\ell+1 & \ldots & (m-1)\ell+1 &  2  & \ell+2 & 2\ell+2 & \ldots & (m-1)\ell+2 & \ldots  &   \ell  & \ldots   & \ell + (m-1)\ell
\end{smallmatrix}\right) .
\end{equation} 
(In Matlab, we can generate $\varpi_{\ell,m}$ as $\varpi_{\ell,m} = \texttt{reshape(reshape($1:\ell*m$,$\ell,m$)',$\ell*m$,1)}$, which intuitively describes $\varpi_{\ell,m}$ as transformation of indices between a column and a row major ordering of a two dimensional array.)
\begin{proposition}\label{PROP:S=KR} 
	The LS problem (\ref{zd:eq:g-S*alpha}) has the following Khatri--Rao structure:
	\begin{itemize}
\item[(i)] $ S = \Pi (R \odot \Vanderm_{\ell,m}^T)=\Vanderm_{\ell,m}^T\odot R$, or, equivalently, $ S(\varpi_{\ell,m},:) = R \odot \Vanderm_{\ell,m}^T$.
\item[(ii)] $(R^* R) \circ (\overline{\Vanderm_{\ell,m} \Vanderm_{\ell,m}^*}) = (R \odot \Vanderm_{\ell,m}^T)^* (R \odot \Vanderm_{\ell,m}^T)=(\Vanderm_{\ell,m}^T\odot R)^* (\Vanderm_{\ell,m}^T\odot R)=S^* S$.
\item[(iii)] For any $m\times m$ matrix $\WW$,  $(\WW \otimes \Id_m)S= (\WW \Vanderm_{\ell,m}^T)\odot R$.
	\end{itemize}
\end{proposition}
\begin{proof}
	Brute--force.
\end{proof}
Since $\Pi$ is orthogonal, the problem reduces to computing the QR factorization of $R \odot \Vanderm_{\ell,m}^T$, and the objective (\ref{zd:eq:g-S*alpha}) can be written as 
\begin{equation}
\| \vec{\g} -
S \vec{\bfalpha}\|_2 \equiv \| \vec{\g} -
\Pi (R \odot \Vanderm_{\ell,m}^T ) \vec{\bfalpha}\|_2 \equiv  \| \Pi^T \vec{\g} -
(R \odot \Vanderm_{\ell,m}^T) \vec{\bfalpha}\|_2\longrightarrow\min .
\end{equation}
Note that $\vec{\g}=\mathrm{vec}(G)$ and that $\Pi^T\vec{\g}=\mathrm{vec}(G^T)$, where $G$ is from (\ref{eq:G}), and $\mathrm{vec}(\cdot)$ denotes the operator that reshapes a matrix by stacking its columns in a tall vector.
Normal equations are attractive in this setting because of avoiding the row dimension $m\ell$ of the Khatri--Rao product.

\subsubsection{Application in antenna array processing}\label{SSS=Antena-Array}
The matrix least squares approximation with the Khatri-Rao structure also appears in other applications. An example is multistatic antenna array processing, where one determines the scattering coefficients $\alpha_i$ by solving 
$$
\| H - G_{rec} \mathrm{diag}(\alpha_i)_{i=1}^{\ell} G_{tr}^T \|_F \longrightarrow \min_{\alpha_i}.
$$
Here $H$ stands for the multistatic data matrix, and the columns  $G_{rec}(:,i)$ and $G_{tr}(:,i)$ are the steering vectors associated with wave propagation between the receiving and transmitting array, respectively, and the $i$th scatterer. We refer the reader to 
 \cite{lev-ari2005}, where the unweighted (i.e. $\WW=\Id$) version of  (\ref{eq:alpha:normal:eq}) is derived as normal equations solution based on the properties of the Khatri--Rao product. In fact, Theorem \ref{TM:weighted_LS} provides a generalization of \cite{lev-ari2005} to weighted least squares, and the material of this section supplies the supporting numerical analysis. 
(The formulas in Theorem \ref{TM:weighted_LS} do not use the fact that $\Vanderm_{\ell,m}$ is Vandermonde matrix.)

\section{An improvement of the sparsity promoting DMDSP}\label{SS=DMDSP-improved}
In the notation of \cite{jovschnicPOF14}, with $\ell=k$ (i.e. using $k$ modes, where $k$ is the dimension of the POD basis), we can write $\Omega^2(\vec{\bfalpha})$ in the unweighted case ($\WW=\Id_m$) as
\begin{equation}
\Omega^2(\vec{\bfalpha}) = \vec{\bfalpha}^* \mathbb{P}\vec{\bfalpha} - \mathbf{q}^*\vec{\bfalpha} - \vec{\bfalpha}^* \mathbf{q} + \|\Sigma_k\|_F^2,\;\;\mbox{where}\;\; \mathbb{P}= (Z_{k}^*Z_{k}) \circ (\overline{\Vanderm_{k,m} \Vanderm_{k,m}^*}),
\end{equation}
and
\begin{equation}
\mathbf{q} = \mathrm{diag}(\overline{\Vanderm_{k,m}V_k\Sigma_k B_k}),
\end{equation}
where $Z_k=U_k B_k$ and $B_k$ is the matrix of the Ritz vectors computed in Line 5 of Algorithm \ref{zd:ALG:DMD}. (If the refined Ritz vectors are used, then $B_k$ is replaced with $W_k$ that is computed in Line 12 of Algorithm \ref{zd:ALG:DMD:RRR}.)
\begin{remark}
{\em
Let us compare this $\mathbf{q}$ and the corresponding 
$\mathbf{q_0} =	[ (\overline{\Vanderm_{k,m}}\circ (Z_{k}^* \X_m))\eb ]$ in relation (\ref{eq:alpha:normal:eq:orig-data}) (with $\ell=k$ and $W=\Id_m$). 
Since $Z_{k}^* \X_m=B_k^*\Sigma_k V_k^*$,  by \cite[Lemma 5.1.3]{hor-joh-91}, we have
$$
\mathbf{q_0}=[ (\overline{\Vanderm_{k,m}}\circ (B_{k}^* \Sigma_k V_k^*))\eb ] = \mathrm{diag}(\overline{\Vanderm_{k,m}} (B_{k}^* \Sigma_k V_k^*)^T = 
\mathrm{diag}(\overline{\Vanderm_{k,m}} \overline{V_k}\Sigma_k \overline{B_{k}}) = \mathbf{q}.
$$
}	
\end{remark}
The constrained problem (\ref{eq:DMDSP:problem}) ($\Omega^2(\vec{\bfalpha})\longrightarrow\min$, $E^T\vec{\bfalpha}=\0$, where the $k\times s$ matrix $E$ is described in \S \ref{SS=DMDSP}) is in the polishing phase of  DMDSP solved via the variation of the Lagrangian (see \cite[Appendix C]{jovschnicPOF14}), which yields the augmented $(k+s)\times(k+s)$ system of equations
\begin{equation}\label{eq:saddle}
\begin{pmatrix}
\mathbb{P} & E \cr E^T & \0
\end{pmatrix} \begin{pmatrix} \vec{\bfalpha} \cr \bfnu \end{pmatrix} = \begin{pmatrix} \mathbf{q} \cr \0\end{pmatrix}
\end{equation}
for the amplitudes $\vec{\bfalpha}$ and the vector of Lagrange multipliers $\bfnu$.
The amplitudes are then explicitly expressed as (see \cite[Appendix C]{jovschnicPOF14} and \cite{dmdsp-matlab})
\begin{equation}\label{eq:DMDSP:alpha}
\vec{\bfalpha} = \begin{pmatrix} \Id_k & \0 \end{pmatrix} \left[\begin{pmatrix}
\mathbb{P} & E \cr E^T & \0 \end{pmatrix}^{-1} \begin{pmatrix} \mathbf{q} \cr \0\end{pmatrix}\right] .
\end{equation}
More efficient procedure is to write $\vec{\bfalpha}$ as $\vec{\bfalpha}=E_{\perp}\vec{\bfbeta}$, where $E_{\perp}\in\C^{k\times (k-s)}$ contains the remaining columns of the identity (complementary to the columns of $E$), and $\vec{\bfbeta}\in\C^{k-s}$ is a new (unconstrained) variable.
The objective function now reads
\begin{equation}
\Omega^2(\vec{\bfalpha})_{E^T\vec{\bfalpha}=\0} \equiv \Omega^2(\vec{\bfbeta}) =
\vec{\bfbeta}^* (E_{\perp}^*\mathbb{P}E_{\perp}) \vec{\bfbeta} - (E_{\perp}^* \mathbf{q})^*\vec{\bfbeta} - \vec{\bfbeta}^* (E_{\perp}^*\mathbf{q}) + \|\Sigma_k\|_F^2 ,
\end{equation}
where $E_{\perp}^*\mathbb{P}E_{\perp}$ is a $(k-s)\times (k-s)$ main submatrix of $\mathbb{P}$, and the explicit solution is
\begin{equation}\label{eq:DMDSP:alpha:new}
\vec{\bfalpha} = E_{\perp} \left[ (E_{\perp}^*\mathbb{P}E_{\perp})^{-1} (E_{\perp}^*\mathbf{q}) \right].
\end{equation}
Schematically, the coefficient matrices of the linear systems in (\ref{eq:DMDSP:alpha}) and (\ref{eq:DMDSP:alpha:new}) can be illustrated as follows:
\begin{equation}
\begin{pmatrix}
\mathbb{P} & E \cr E^T & \0
\end{pmatrix} = \left(\begin{smallmatrix}
* & * & * & * & * & * & * & * & \mathbf{1} & 0 & 0 & 0 & 0 & 0 \cr 
* & \boldsymbol{\circledast} & * & * & \boldsymbol{\circledast} & * & * & * & 0 & 0 & 0 & 0 & 0 & 0 \cr
* & * & * & * & * & * & * & * & 0 & \mathbf{1} & 0 & 0 & 0 & 0 \cr
* & * & * & * & * & * & * & * & 0 & 0 & \mathbf{1} & 0 & 0 & 0 \cr
* & \boldsymbol{\circledast} & * & * & \boldsymbol{\circledast} & * & * & * & 0 & 0 & 0 & 0 & 0 & 0 \cr
* & * & * & * & * & * & * & * & 0 & 0 & 0 & \mathbf{1} & 0 & 0 \cr
* & * & * & * & * & * & * & * & 0 & 0 & 0 & 0 & \mathbf{1} & 0 \cr
* & * & * & * & * & * & * & * & 0 & 0 & 0 & 0 & 0 & \mathbf{1} \cr
\mathbf{1} & 0 & 0 & 0 & 0 & 0 & 0 & 0 & \mathbf{0} & \mathbf{0} & \mathbf{0} & \mathbf{0} & \mathbf{0} & \mathbf{0} \cr
0 & 0 & \mathbf{1} & 0 & 0 & 0 & 0 & 0 & \mathbf{0} & \mathbf{0} & \mathbf{0} & \mathbf{0} & \mathbf{0} & \mathbf{0} \cr
0 & 0 & 0 & \mathbf{1} & 0 & 0 & 0 & 0 & \mathbf{0} & \mathbf{0} & \mathbf{0} & \mathbf{0} & \mathbf{0} & \mathbf{0} \cr
0 & 0 & 0 & 0 & 0 & \mathbf{1} & 0 & 0 & \mathbf{0} & \mathbf{0} & \mathbf{0} & \mathbf{0} & \mathbf{0} & \mathbf{0} \cr
0 & 0 & 0 & 0 & 0 & 0 & \mathbf{1} & 0 & \mathbf{0} & \mathbf{0} & \mathbf{0} & \mathbf{0} & \mathbf{0} & \mathbf{0} \cr
0 & 0 & 0 & 0 & 0 & 0 & 0 & \mathbf{1} & \mathbf{0} & \mathbf{0} & \mathbf{0} & \mathbf{0} & \mathbf{0} & \mathbf{0} \cr
\end{smallmatrix}\right),\;\;
E_{\perp}^*\mathbb{P}E_{\perp} = \left( \begin{smallmatrix} \boldsymbol{\circledast} &  \boldsymbol{\circledast} \cr \boldsymbol{\circledast} & \boldsymbol{\circledast} \end{smallmatrix}\right),\;\;
E_{\perp} = \left( \begin{smallmatrix} 
0 & 0 \cr
\mathbf{1} & 0 \cr
0 & 0 \cr
0 & 0 \cr
0 & \mathbf{1} \cr
0 & 0 \cr
0 & 0 \cr
0 & 0
\end{smallmatrix}\right).
\end{equation}
\begin{remark}
	{\em 
		To appreciate the difference between (\ref{eq:DMDSP:alpha}) and (\ref{eq:DMDSP:alpha:new}), consider for example  $k=m=1200$, and  only $\ell=30$ modes. The augmented system in (\ref{eq:DMDSP:alpha})  is of dimension $1200+1170=2370$, while the same result is obtained from the system (\ref{eq:DMDSP:alpha:new}) of dimension $30$. Given the cubic complexity of the solution of linear systems, the speedup of (\ref{eq:DMDSP:alpha:new}) over (\ref{eq:DMDSP:alpha}) is considerable. 
		It should be noted that (\ref{eq:saddle}) is a particular case of the saddle point problem that is solved by specially tailored methods; for an overview we refer to \cite{benzi_golub_liesen_2005}; see also e.g. \cite{Gould-Proj_Krylov-Saddle-simax-2014}. However, our proposed alternative (\ref{eq:DMDSP:alpha:new}) is simpler and much more efficient. 
	}
\end{remark}

\begin{remark}
	{\em
		The same scheme applies in the more general case with arbitrary full column rank matrix $E$. The only technical difference is in constructing the complement $E_{\perp}$. Let 
		$
		E = Q\left(\begin{smallmatrix} R \cr \0\end{smallmatrix}\right) \equiv \begin{pmatrix} Q_1 & Q_2 \end{pmatrix} \left(\begin{smallmatrix} R \cr \0\end{smallmatrix}\right)
		$	
		be the full QR factorization. It is then clear that we can take $E_{\perp}=Q_2$, and the above procedure applies verbatim.
	}
\end{remark}

\section{Sensitivity analysis of normal equations method}\label{S=Normal-Sensitivity}
Although elegant and efficient, the formula (\ref{eq:alpha:normal:eq}) has a drawback, typical for normal equation based LS solution -- it squares the condition number ($\kappa_2(R^* R)=\kappa_2(R)^2$ and $\kappa_2(\overline{\Vanderm_{\ell,m}\WW^2 \Vanderm_{\ell,m}^*})=\kappa_2(\Vanderm_{\ell,m}\WW)^2$) and uses a Cholesky factorization based solver with the coefficient matrix $C\equiv (R^* R) \circ (\overline{\Vanderm_{\ell,m}\WW^2 \Vanderm_{\ell,m}^*})$. 
In theory, by a Schur theorem \cite[Theorem 5.2.1]{hor-joh-91}, the Hadamard product of a positive definite matrix and a positive semidefinite matrix with positive diagonal is positive definite and thus possesses the Cholesky factor.

In extremely ill-conditioned cases, however, it can happen that both computed and stored matrices $R^*R$ and $\Vanderm_{\ell,m} \WW^2 \Vanderm_{\ell,m}^*$ are exactly singular (or even indefinite) so that the Cholesky factorization of $C$ might fail. 
Recall that, in finite precision computation, the Cholesky factorization may fail even if the matrix stored in the machine memory is exactly positive definite. Moreover, the factorization may even succeed with an indefinite matrix on input. We refer to \cite{demmel-89-Cholesky} for details on computing the Cholesky factorization in floating point arithmetic. 


Hence, the formula (\ref{eq:alpha:normal:eq}) should be deployed with great care, and its sensitivity must be well understood as it is used in a variety of applications. 

\noindent The following example is contrived, to illustrate the above discussion. 
\begin{example}\label{EX:bad-normal-1}
	{\em 
	(All computations done in Matlab R2015.a, with the double precision roundoff unit $\roff=\texttt{eps=2.220446049250313e-16}$.)
Let $\ell=3$, $m=4$, $\xi=\sqrt{\roff}$, $\lambda_1 = \xi$, $\lambda_2=2\xi$, $\lambda_3=0.2$, so that the Vandermonde section $\Vanderm_{\ell,m}$ equals
{
$$
\Vanderm_{\ell,m} = \left(\begin{smallmatrix}
1.000000000000000e+00  &   1.490116119384766e-08   &  2.220446049250313e-16  &   3.308722450212111e-24 \cr 
1.000000000000000e+00  &   2.980232238769531e-08   &  8.881784197001252e-16   &  2.646977960169689e-23 \cr 
1.000000000000000e+00  &   2.000000000000000e-01   &  4.000000000000001e-02   &  8.000000000000002e-03
 \end{smallmatrix}\right) ,
$$  
}
and let the triangular factor $R$ be
$$
R = \begin{pmatrix} 1 & 1 & 1 \cr 0 & \xi/2 & \xi \cr 0 & 0 & \xi	
	\end{pmatrix} = \left( \begin{smallmatrix}
	1.000000000000000e+00  &   1.000000000000000e+00  &   1.000000000000000e+00\cr
	0   &  7.450580596923828e-09  &   1.490116119384766e-08 \cr
	0                &         0   &  1.490116119384766e-08
	\end{smallmatrix}\right) .
	$$
Set $\WW=\Id_m$. From the singular values of $\Vanderm_{\ell,m}$ and $R$, computed as
{\small 
$$ 
\begin{array}{c|c|c}
  &   \sigma_i(\Vanderm_{\ell,m}) & \sigma_i(R) \cr\hline
 1&    1.736092504099537e+00 &    1.732050807568878e+00 \cr 
 2&    1.662733207986230e-01  &   1.555891180151553e-08 \cr
 3&    2.105723035894610e-09   &  4.119745457168918e-09 \cr
     \end{array},
$$
we see that their condition numbers are of the order of  $1/\sqrt{\roff}$, which leaves the possibility of computation (involving $R$ and $\Vanderm_{\ell,m}$) with an $O(\sqrt{\roff})$ accuracy -- in the IEEE 64 bit arithmetic this means seven to eight decimal places. Moreover, both $\Vanderm_{\ell,m}$ and $R$ are of full rank and the closest rank deficient matrices are at distances $\sigma_3(\Vanderm_{\ell,m}) > 10^{-9}$ and $\sigma_3(R) >10^{-9}$, respectively.
}
Further, since both matrices are entry-wise nonnegative, both $\Vanderm_{\ell,m}\Vanderm_{\ell,m}^*$ and $R^* R$ are computed to nearly full machine precision\footnote{In this example, there are no proper subtractions and each entry in both matrices is approximated with the corresponding floating point number up to an $O(\roff)$ relative error -- practically the best one can hope for.}; the same holds for $C= (R^* R) \circ (\overline{\Vanderm_{\ell,m} \Vanderm_{\ell,m}^*})$. Although all these three matrices  are by design positive definite and computed in each entry to full machine precision, none of them is numerically positive definite -- the Cholesky factorization fails on each of them: \\

\begin{minipage}[t]{0.5\textwidth}
\begin{verbatim}
>> chol(Vlm*Vlm')
Error using chol
Matrix must be positive definite.

>> chol(R'*R)
Error using chol
Matrix must be positive definite.
\end{verbatim}
\end{minipage}
\begin{minipage}[t]{0.5\textwidth}
\begin{verbatim}
>> chol((R'*R).*(Vlm*Vlm'))
Error using chol
Matrix must be positive definite.
\end{verbatim}
\end{minipage}
\ \\
Hence, the normal equation solver (which assumes positive definite linear system) might fail; and even if it succeeded the result might be unacceptably  inaccurate.
}
\end{example}

On the other hand, the unweighted form ($\WW=\Id$) of the formula (\ref{eq:alpha:normal:eq}) has been successfully used in the computational DMD framework, despite the fact that it is based on normal equation (an ill-advised approach) that involves potentially ill-conditioned matrices. We offer an explanation and provide a way to estimate a priori whether this method can produce sufficiently accurate result. 

\subsection{Condition number estimates}\label{SSS=Cond-est}
Based on \cite{demmel-89-Cholesky}, we know that floating point Cholesky factorization $C=LL^*$ ($L$ lower triangular with positive diagonal) of $C$ is feasible if\footnote{Actually, if no additional structure is assumed, here we have an "if and only if". This means that, in absence of additional properties such as sparsity or sign distribution of matrix entries, the failure of the Cholesky decomposition means that the matrix cannot be considered numerically positive definite.} the matrix $C_s= (c_{ij}/\sqrt{c_{ii}c_{jj}})_{i,j=1}^{\ell}$ is well conditioned. 
Further, if $\widetilde{L}$ is the computed Cholesky factor, then $\widetilde{L}\widetilde{L}^* = C + \delta C$, where the backward error $\delta C = (\delta c_{ij})_{i,j=1}^\ell$ of the computed factorization is such that 
\begin{equation}\label{eq:dC/D}
\max_{i,j} \frac{|\delta c_{ij}|}{\sqrt{c_{ii}c_{jj}}} \leq O(\ell) \roff .
\end{equation}
If we set $D_C=\mathrm{diag}(\sqrt{c_{ii}})_{i=1}^\ell$, then the relation above can be written as 
\begin{equation}\label{eq:Chol-backward}
\widetilde{L}\widetilde{L}^* = C + \delta C = D_C ( C_s + D_C^{-1} \delta C D_C^{-1} ) D_C, \;\;
C_s = D_C^{-1} C D_C^{-1},
\end{equation}
where the entries of $D_C^{-1} \delta C D_C^{-1}$ are estimated by (\ref{eq:dC/D}). Note that $C_s$ has unit diagonal and that all its off-diagonal entries are in absolute value below one. 

By \cite[Theorem 3.1]{drm-oml-ves-94}, we can write $\widetilde{L} = L (\Id+\Gamma)$, where $\Id+\Gamma$ is lower triangular with positive diagonal and the size of the multiplicative error (note: $\delta L \equiv \widetilde{L}-L = L \Gamma$) can be bounded by 
\begin{equation}\label{eq:Cholesky:Gamma}
\|\Gamma\|_2 \leq O(\ell\log_2 \ell)\roff \| C_s^{-1}\|_2 \leq O(\ell\log_2 \ell) \roff \kappa_2(C_s).
\end{equation}
Further, if we solve the linear system $ C x = b\neq \0$ using the Cholesky factor in the forward and backward substitutions, then the computed solution $\widetilde{x}$ satisfies (see \cite[Theorem 2.1]{demmel-89-Cholesky})
\begin{equation}
\frac{\| D_C ( \widetilde{x} - C^{-1}b ) \|_2}{\| D_C \widetilde{x}\|_2} \leq g(\ell) \roff \kappa_2(C_s),
\end{equation}
where $g(\ell)$ is modest function of the dimension. Note that this implies component-wise error bound for each $\widetilde{x}_i \neq 0$:
\begin{equation}
\frac{|\widetilde{x}_i - (C^{-1}b)_i|}{|\widetilde{x}_i|} \leq 
\underbrace{\left[ \frac{\| D_C \widetilde{x}\|_2}{\sqrt{c_{ii}}|\widetilde{x}_i|} \right]}_{\geq 1} g(\ell) \roff \kappa_2(C_s).
\end{equation}
For further discussion we refer to \cite{demmel-89-Cholesky}.

Hence, it is the scaled condition number $\kappa_2(C_s)$ that determines the sensitivity to perturbations and the accuracy of the computed amplitudes, and not $\kappa_2(C)$. 
This is important because of the following theorem by Van der Sluis \cite{slu-69}:
\begin{theorem}\label{TM:VanDerSluis}
	Let $H\in\C^{n\times n}$ be positive definite Hermitian matrix, $D_H=\mathrm{diag}(\sqrt{H_{ii}})$ and $H_s = D_H^{-1} H D_H^{-1}$.
	Then
	$
	\kappa_2(H_s)\leq n \min_{D=diag}\kappa_2(D H D).
	$
\end{theorem}
\begin{remark}\label{REM:Scond}
{\em 
Note that $\kappa_2(C_s)$ can be at most $\ell$ times larger, and that it can be much smaller than $\kappa_2(C)$. 
}
\end{remark}
In our case, $C = A\circ B$, with $A=R^* R$, $B=\overline{\Vanderm_{\ell,m}\WW^2 \Vanderm_{\ell,m}^*}$, and it is important to understand how $\kappa_2(C_s)$ depends on $A$ and $B$.

\begin{theorem}\label{TM:Cond:A*B}
Let $A$ ad $B$ be Hermitian positive semidefinite matrices with positive diagonal entries, and let $C=A\circ B$. If $A_s=(a_{ij}/\sqrt{a_{ii}a_{jj}})$, 	$B_s=(b_{ij}/\sqrt{b_{ii}b_{jj}})$, $C_s= (c_{ij}/\sqrt{c_{ii}c_{jj}})$, then
\begin{equation}
\max( \lambda_{\min}(A_s), \lambda_{\min}(B_s)) \leq \lambda_i(C_s) \leq
\min( \lambda_{\max}(A_s), \lambda_{\max}(B_s)).
\end{equation}
In particular, $\|C_s^{-1}\|_2 \leq \min( \|A_s^{-1}\|_2, \|B_s^{-1}\|_2)$ and  $\kappa_2(C_s)\leq \min( \kappa_2(A_s), \kappa_2(B_s) )$. If $A$ or $B$ is diagonal, all inequalities in this theorem become equalities.
\end{theorem}
\begin{proof}
The key observation is that $C_s=A_s\circ B_s$, and the proof completes by invoking \cite[Theorem 5.3.4]{hor-joh-91}, which states the following general property of the Hadamard product
$$
\lambda_{\min}(A_s)\lambda_{\min}(B_s)\leq \min_{i} (A_s)_{ii} \lambda_{\min}(B_s) \leq \lambda_i(C_s) \leq \max_{i} (A_s)_{ii} \lambda_{\max}(B_s) \leq \lambda_{\max}(A_s)\lambda_{\max}(B_s) ,
$$	
in which we can also swap the roles of $A_s$ and $B_s$. Finally, note that $(A_s)_{ii}=1$ for all $i$.
\end{proof}

\begin{remark}\label{REM:k(C_s)-improvement}
	{\em
		It should be intuitively clear that the Hadamard product $C=A\circ B$ of a positive definite and a positive semidefinite matrix with nonzero diagonal should not worsen the scaled condition number, i.e. that $\kappa_2(C_s)$ is expected to be better than both $\kappa_2(A_s)$ and $\kappa_2(B_s)$.
		Indeed, $C_s$ has unit diagonal and the off-diagonal entries are 
		$$
		(C_s)_{ij} = \frac{c_{ij}}{\sqrt{c_{ii}c_{jj}}} = \frac{a_{ij}}{\sqrt{a_{ii}a_{jj}}}
		\frac{b_{ij}}{\sqrt{b_{ii}b_{jj}}} = (A_s)_{ij} (B_s)_{ij} ,
		$$
		where by the (semi)definiteness of $A$ and $B$ both factors on the right hand side are in modulus below one. That is, the Hadamard products increases the dominance of the diagonal of $C_s$ over any off-diagonal position, as compared to the dominance of the corresponding diagonal entries in $A_s$ and $B_s$. Hence, it is possible that $\kappa_2(C_s)\ll \min( \kappa_2(A_s), \kappa_2(B_s) )$.
	}	
\end{remark}
In the following example we illustrate such a situation, where $R^* R$ and $\overline{\Vanderm_{\ell,m}\Vanderm_{\ell,m}^*}$ are so close to the boundary of the cone of positive definite matrices that they numerically appear as indefinite, but their  Hadamard product posseses numerical Cholesky factor.
\begin{example}\label{EX:bad-normal-2}
	{\em
		With the notation of Example \ref{EX:bad-normal-1}, we use the same $\Vanderm_{\ell,m}$ but change the definition of $R$ to 
		$$
R = \begin{pmatrix} 1 & 1 & 1 \cr 0 & \xi & \xi \cr 0 & 0 & \xi/2	
\end{pmatrix} = \left( \begin{smallmatrix}	
1.000000000000000e+00  &   1.000000000000000e+00   &   1.000000000000000e+00 \cr
0  &   1.490116119384766e-08  &   1.490116119384766e-08 \cr
0           &              0   &  7.450580596923828e-09
\end{smallmatrix}\right) .
$$
If we repeat the experiment with the Cholesky factorizations, we obtain
 \\

\begin{minipage}[t]{0.5\textwidth}
\begin{verbatim}
>> chol(Vlm*Vlm')
Error using chol
Matrix must be positive definite.
\end{verbatim}
\end{minipage}
\begin{minipage}[t]{0.5\textwidth}
\begin{verbatim}
>> chol(R'*R)
Error using chol
Matrix must be positive definite.
\end{verbatim}
\end{minipage}

\begin{verbatim}
>> T0 = chol((R'*R).*(Vlm*Vlm'))
T0 =
1.000000000000000e+00     1.000000000000000e+00     1.000000002980232e+00
                    0     1.490116119384765e-08     1.999999880790710e-01
                    0                         0     4.079214149695062e-02
\end{verbatim}
The condition number of $\texttt{T0}$ is estimated to $8.2\cdot 10^8$, and its inverse is used in backward and forward substitutions when solving the normal equations, see (\ref{eq:alpha:normal:eq}).	

Note, however, that in this case $\kappa_2(C_s) \roff > 2$, and perturbation theory \cite{demmel-89-Cholesky}, \cite{drm-oml-ves-94} provides no guarantee for the accuracy of the computed Cholesky factor. (In fact, in Example \ref{EX:bad-normal-3} below, we argue that $\texttt{T0}$ is a pretty bad approximation.)
}
\end{example}

\begin{corollary}\label{COR:TM:Cond:A*B}
Let $C\equiv (R^* R) \circ (\overline{\Vanderm_{\ell,m} \WW^2 \Vanderm_{\ell,m}^*})$, 
$C_s= (c_{ij}/\sqrt{c_{ii}c_{jj}})$. Further, let $R = R_{c} \Delta_r$ and $\Vanderm_{\ell m}\WW = \Delta_v (\Vanderm_{\ell m}\WW)_{r}$ with diagonal scaling matrices $\Delta_r$ and $\Delta_v$ such that $R_{c}$ has unit columns and 
$(\Vanderm_{\ell m}\WW)_{r}$ has unit rows (in Euclidean norm). Then
$$
\kappa_2(C_s) \leq \min ( \kappa_2(R_c)^2, \kappa_2((\Vanderm_{\ell, m}\WW)_{r})^2).
$$	
\end{corollary}
\begin{proof}
Note that, with the notation $A=R^* R$, $B=\overline{\Vanderm_{\ell,m}\WW^2 \Vanderm_{\ell,m}^*}$,	we can apply Theorem \ref{TM:Cond:A*B}, where we have
$A_s = R_c^* R_c$, $B_s=(\Vanderm_{\ell m}\WW)_{r}(\Vanderm_{\ell, m}\WW)_{r}^*$.
\end{proof}
\begin{example}\label{EX:kappa(C_s)}
	{\em
According to Theorem \ref{TM:VanDerSluis} and Remark \ref{REM:Scond}, $\kappa_2(C_s)$ is potentially much smaller than $\kappa_2(C)$, and, by the discussion in \S \ref{SSS=Cond-est}, the results may be much better than predicted by the classical perturbation theory based on $\kappa_2(C)$. And, we can very precisely determine whether the normal equation solution will succeed  by computing/estimating $\kappa_2(C_s)$. 

This can be illustrated e.g. with the simulation data of a 2D model obtained by depth averaging the Navier--Stokes equations for a shear flow in a thin layer of electrolyte suspended on a thin lubricating layer of a dielectric fluid; see \cite{Tithof-2017JFM-Bif-Q2D-Kolmogorov}, \cite{Suri-2014PhFl-Velocity-Kolmogorov-flow} for more detailed description of the experimental setup and numerical simulations.\footnote{We thank Michael Schatz, Balachandra Suri, Roman Grigoriev and Logan Kageorge from the Georgia Institute of Technology for providing us with the data.}		
		
The  data\footnote{We used this dataset in \cite{drmac-mezic-mohr-DFT-arxiv-2018}	for testing a new algorithm for representing the snapshots using thr Koopman modes.	} consists of $n_t=1201$ snapshots of dimensions $n_x\times n_y$ ($n_x=n_y=128$), representing the (scalar) vorticity field.  The $n_x\times n_y \times n_t$ tensor is unfolded into $n_x\cdot n_y \times n_t$ matrix $(\f_1,\ldots, \f_{n_t})$, and $\X_m$ is of dimensions $16384\times 1200$.

After performing a DMD analysis, we have computed the values of the condition numbers of $C\equiv (R^* R) \circ (\overline{\Vanderm_{\ell,m} \Vanderm_{\ell,m}^*})$ and $C_s$ as follows:
$$
\kappa_2(C) > 10^{87} \gg \kappa_2(C_s) \approx \texttt{8.504071414461372e+01}.
$$	
Using (\ref{eq:Cholesky:Gamma}), we can conclude that the Cholesky factor of $C$ can be computed to high accuracy, despite the fact that $\kappa_2(C) \gg 1/\roff$. A closer look reveals that the high condition number is due to grading, i.e. it is in the diagonal $D_C$. More precisely, as a result of the Ritz values being from both sides of the unit circle, the diagonal entries of $C$ vary over several orders of magnitude, see Figure \ref{zd:fig:diag_of_C}.
\begin{figure}[H]
	\begin{minipage}[c]{0.70\textwidth}
		\includegraphics[width=\textwidth, height=0.50\textwidth]{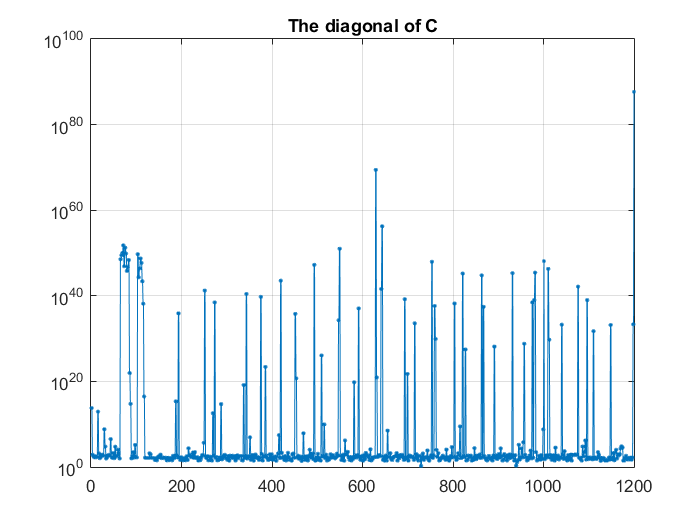}
	\end{minipage}\hfill
	\begin{minipage}[c]{0.30\textwidth}
		\caption{\label{zd:fig:diag_of_C} The diagonal entries of $C$ vary from 
			$\texttt{1.682014829394577e+00}$ to $\texttt{4.574179501782303e+87}$. The absolute values of the computed Ritz values go (roughly) from  $0.63$ to $1.08$.
		} 
	\end{minipage}
\end{figure}
Following  the discussion in Remark \ref{REM:k(C_s)-improvement} and Examples \ref{EX:bad-normal-1} and \ref{EX:bad-normal-2}, the bound of Corollary \ref{COR:TM:Cond:A*B} is in this example (luckily) an extreme overestimate because 
	$\kappa_2(R_c)^2 =\texttt{3.052938825507679e+13}$, $\kappa_2((\Vanderm_{\ell, m})_{r})^2 = \texttt{9.620903374928321e+14}$.
	However, Example \ref{EX:bad-normal-1} illustrates the danger of potential failure of the formula (\ref{eq:alpha:normal:eq}). 
	}
	\end{example}
	
\begin{remark}
{\em
In the DMDSP, the augmented Lagrangian method works with  $C+\rho I$, with suitable parameter $\rho>0$. The Cholesky factor of $C+\rho I$ is used to repeatedly solve linear systems of equations. With any moderate value of $\rho$, the diagonal of $C$	in Example \ref{EX:kappa(C_s)} will behave similarly as shown in Figure \ref{zd:fig:diag_of_C}.
}	
\end{remark}
The condition number $\kappa_2(C_s)$ from Corollary \ref{COR:TM:Cond:A*B} will determine the accuracy if we can compute $C$ so that the error $\delta_0 C$ in computing $C$ explicitly is such that $|\delta_0 c_{ij}|/\sqrt{c_{ii}c_{jj}}$ is small for all $i$, $j$.

\begin{proposition}
Let $A=X^* X$, $B=Y^* Y$ with $X\in\C^{m_x\times\ell}$, $Y\in\C^{m_y\times\ell}$, and let $C=A\circ B$, and $\widetilde{C}=C+\delta_0 C = computed(computed(X^*X)\circ computed(Y^* Y))$. Then, for all $i$, $j$, 
$$
|\delta_0 c_{ij}| \leq (O(m_x\roff)+O(m_y\roff)+O(m_x m_y\roff^2)) \sqrt{c_{ii}c_{jj}} .
$$
\end{proposition}	
\begin{proof}
We follow the standard steps of floating point error analysis:
\begin{eqnarray*}
computed(X^* X) = A + \delta A,\;\;|\delta A|\leq O(m_x\roff) |X^*||X| ;\;\;
|\delta a_{ii}| \leq O(m_x\roff) a_{ii},
\end{eqnarray*}	
and,  by the Cauchy-Schwarz inequality, $|\delta a_{ij}| \leq O(m_x\roff) \sqrt{a_{ii}a_{jj}}$.  An analogous claim holds for $computed(Y^* Y)=B+\delta B$. Altogether, 
$$
\widetilde{c}_{ij}=(a_{ij}+\delta a_{ij})(b_{ij}+\delta b_{ij})(1+\epsilon_{ij})
= (c_{ij} + \delta a_{ij}b_{ij} + a_{ij}\delta b_{ij} + \delta a_{ij}\delta b_{ij})(1+\epsilon_{ij})
$$
where (since $c_{ii}=a_{ii}b_{ii}$ and $|b_{ij}|\leq\sqrt{b_{ii}b_{jj}}$)
$$
|\delta a_{ij}b_{ij}|\leq O(m_x\roff)\sqrt{a_{ii}a_{jj}}\sqrt{b_{ii}b_{jj}}=O(m_x\roff)\sqrt{c_{ii}c_{jj}}
$$
The remaining error terms are bounded in the same way.
\end{proof}
Hence, in a concrete finite precision computation, we will work with $\widetilde{C}=C+\delta_0 C$, so that the backward error (\ref{eq:dC/D}, \ref{eq:Chol-backward}) applies to $\widetilde{C}$. The scaled condition number $\kappa_2(\widetilde{C}_s)$ ($\widetilde{C}_s=(\widetilde{c}_{ij}/\sqrt{\widetilde{c}_{ii}\widetilde{c}_{jj}})$) will be moderate if $\kappa_2(C_s)$ is moderate, so all conclusions apply to $\widetilde{C}$ as well. Since both $\delta_0 C$ and $\delta C$ are of the same type, the overall perturbation in $C$ and its Cholesky factor is bounded in the same way. We omit the details for the sake of brevity.

\subsubsection{A decision tree}
Equipped with the results from \S \ref{SSS=Cond-est}, we can devise a strategy that chooses the most efficient procedure to deliver the output to the accuracy warranted by the data on the input. Our goal is to develop a reliable software tool that is capable of computing to reasonable accuracy even in the most extreme cases. 

Since the scaled condition number $\kappa_2(C_s)$ has been identified as the key parameter, we can safely proceed with solving the normal equations if we know a priori that it is moderate.
One way to establish that is to use the the upper triangular factor $R$ in the QR factorization $Z_{\ell}=Q\left(\begin{smallmatrix} R \cr 0\end{smallmatrix}\right)$ of the approximate eigenvectors (e.g. Ritz vectors or the refined Ritz vectors) that are usually computed as normalized, i.e. $\|Z_{\ell}(:,i)\|_2=1$ and $R_c\equiv R$. Independent of that normalization, by Corollary \ref{COR:TM:Cond:A*B}, the condition number that matters is $\kappa_2(R_c)\equiv \kappa_2((Z_{\ell})_c)$, where $(Z_{\ell})_c$ denotes the matrix $Z_{\ell}$ after normalizing its columns. 

Recall that the condition number of an $\ell\times \ell$ triangular matrix can be efficiently estimated at an $O(\ell^2)$ cost; see for instance the subroutine \texttt{XPOCON()} in \texttt{LAPACK}.
Moreover, since $Z_{\ell}=U_k \widetilde{B}_{\ell}$ (see Remark \ref{REM:POD-reconstruct}) and since $R$ can be computed directly from the $O(k\ell^2)$ QR factorization of the $k\times\ell$ matrix $\widetilde{B}_{\ell}$, we can  estimate $\kappa_2(R_c)$ at the cost of $O(k\ell^2)$. If $R$ is already available (computed for some other use), then the cost of estimating $\kappa_2(R_c)$ is only $O(\ell^2)$ and in that case we consider the estimate available as well. 
Hence, if one estimates that $\kappa_2(R_c)^2 \ll 1/\roff$, then one can safely use the normal equation solution (\ref{eq:alpha:normal:eq}). 

If an estimate for $\kappa_2(R_c)^2$ is not available (e.g. $R$ not computed) or if one concludes that $\kappa_2(R_c)^2\gtrapprox 1/\roff$, then one goes on and computes $C$ and an estimate for $\kappa_2(C_s)$, which is then tested against the threshold $1/\roff$.
We can organize this in a decision tree as in Algorithm \ref{zd:ALG:Decision-Tree}. It is designed for the case when the dimensions are sufficiently large and the efficiency of the formula (\ref{eq:alpha:normal:eq}) is desirable, but not at any price -- it is deployed only if it can warrants certain level of accuracy.

\begin{algorithm}[H]
	\caption{Decision tree for selecting a solution method for (\ref{zd:eq:g-S*alpha})}
	\label{zd:ALG:Decision-Tree}
	\begin{algorithmic}[1]
		\REQUIRE \  \\		
		\begin{itemize} 
			\item $R$, $\Lambda$, $\ell$, $m$
			\item Tolerance level $tol \in (1,1/\roff)$ for acceptable condition number estimate.		
		\end{itemize}
		%
\IF{ $R$  available and $\kappa_2(R_c)^2 \leq tol$ }
\STATE Solve (\ref{zd:eq:g-S*alpha}) by normal equations, $\vec{\bfalpha} = [ (R^* R) \circ (\overline{\Vanderm_{\ell,m} \WW\Vanderm_{\ell,m}^*}) ]^{-1} [ (\overline{\Vanderm_{\ell,m}\WW}\circ (Z_{\ell}^* \X_m\WW))\eb ]$
\ELSE
\STATE Compute $C= (Z_{\ell}^*Z_{\ell}) \circ (\overline{\Vanderm_{\ell,m}\WW^2 \Vanderm_{\ell,m}^*})$ and estimate $\kappa_2(C_s)$
\IF{$\kappa_2(C_s)\leq tol$}
\STATE Solve (\ref{zd:eq:g-S*alpha}) by normal equations, $\vec{\bfalpha} = [ (Z_{\ell}^*Z_{\ell}) \circ (\overline{\Vanderm_{\ell,m}\WW^2 \Vanderm_{\ell,m}^*}) ]^{-1} [ (\overline{\Vanderm_{\ell,m}\WW}\circ (Z_{\ell}^* \X_m\WW))\eb ]$
\ELSE
\STATE \framebox{Solve (\ref{zd:eq:g-S*alpha}) using QR factorization based solver, without squaring the condition number.}
\ENDIF
\ENDIF
	\end{algorithmic}
\end{algorithm}
\noindent In the next section, we discuss \framebox{Line 8} in Algorithm \ref{zd:ALG:Decision-Tree}.

\begin{example}\label{EX:bad-normal-3}
	{\em
(Continuation of Example \ref{EX:bad-normal-2})
		From Proposition \ref{PROP:S=KR}, the (upper triangular) Cholesky factor of $(R^* R) \circ (\overline{\Vanderm_{\ell,m} \Vanderm_{\ell,m}^*})$ can be equivalently obtained from the QR factorization of $S=\Vanderm_{\ell,m}^T\odot R$. If we compute $S$ explicitly and compute its QR factorization, normalized to have positive diagonal, then
	\begin{verbatim}
	>> [Q,T] = qr(S,0) ; T = diag(sign(diag(T)))*T
	T =
	1.000000000000000e+00     1.000000000000000e+00     1.000000002980232e+00
	                    0     2.107342425544703e-08     1.414213575017149e-01
	                    0                         0     1.471869344809795e-01	
	\end{verbatim}
	Note that the difference between (theoretically identical matrices) $\texttt{T}$ and $\texttt{T0}$, computed in Example \ref{EX:bad-normal-2},
	 is significant. Since the condition number of $S$ is $\kappa_2(S)\approx 1.6 \cdot 10^8$, we know that $\texttt{T}$ is provably more accurate than $\texttt{T0}$.  
	
	The consequence of this difference in applications of the computed triangular factor is easily illustrated by the following simple computation (that is used for solving (\ref{eq:alpha:normal:eq}))
	\begin{verbatim}
	>> [T0\(T0'\ones(3,1)) T\(T'\ones(3,1))]
	ans =
	 2.703847231205138e+01     2.000000063280420e+00
	-2.603847410305737e+01    -1.000000063280429e+00
	 1.791005994071227e-06     9.429518722743198e-15
	\end{verbatim} 
}
\end{example}

\noindent Although small dimensional and entirely synthetic, Examples  \ref{EX:bad-normal-2} and \ref{EX:bad-normal-3} should be convincing enough to call for the development of an algorithm that solves the LS problem without the explicitly computed $(R^* R) \circ (\overline{\Vanderm_{\ell,m} \Vanderm_{\ell,m}^*})$. We tackle this problem in the next section. 
For the sake of brevity, we will consider the case $\WW=\Id_m$. The more involved general weighted case is left for the future work.

\section{QR factorization based solvers}\label{SS=QR-LS-solver}
\noindent A numerically more robust procedure for solving the LS problem (\ref{zd:eq:g-S*alpha}) needs the QR factorization of the $\ell m \times \ell$ matrix $S$, with the complexity of $O(m\ell^3)$ if an off-the-shelf procedure is deployed, e.g. \texttt{qr} from Matlab or \texttt{xGEQRF}, \texttt{xGEQP3} from LAPACK. This $O(m\ell^3)$ factor will then dominate the overall cost of the LS solution, even if one decides to solve (\ref{zd:eq:g-S*alpha}) using the SVD of $S$. (Since $S$ is tall and skinny, the SVD of $S$ is computed more efficiently if the computation  starts with the QR factorization and proceeds with the SVD of the triangular factor.)  

Squaring the condition number of $S$ in the explicit formula 
(\ref{eq:alpha:normal:eq}) representing (\ref{eq:S-dagger}) is avoided if we use the tall QR factorization $S=Q_S R_S$; then\footnote{In this section we use the notation from \S \ref{SS=Normal_Eq_solution}.} 
\begin{equation}\label{eq:RQg}
\vec{\bfalpha}=R_S^{-1}(Q_S^*\vec{\g}). 
\end{equation}
Note that for computing $Q_S^*\vec{\g}$ we do not need the explicitly formed factor $Q_S$. Indeed, since $Q_S^*$ is built from a sequence of elementary unitary matrices (Householder reflectors or Givens rotations) so that $Q_S^* S = R_S$, it suffices to apply the same matrices to $\vec{\g}$ as it was the $(\ell+1)$st column of $S$. 
\subsection{Corrected semi-normal approach}\label{SS=Corrected-Semi-Normal}
\noindent If we want to avoid the additional cost for computing $Q_S^* \vec{\g}$, in particular in view of the  efficient formula (\ref{eq:G}) for $S^*\vec{\g}$, then we can settle for using the QR factorization of $S$ only for implicit computing of the Cholesky factor of $(R^* R) \circ (\overline{\Vanderm_{\ell,m} \Vanderm_{\ell,m}^*})$ in (\ref{eq:alpha:normal:eq}), to avoid the problems illustrated in Examples \ref{EX:bad-normal-1} and \ref{EX:bad-normal-2}. Then, the seminormal solution \cite{BJORCK198731} is
\begin{equation}\label{eq:alpha:new}
\vec{\bfalpha}=R_S^{-1}(R_S^{-*}(S^*\vec{\g})).
\end{equation}
This method only partially alleviates the problem of ill-conditioning. It computes more accurate triangular factor than  the pure normal equations approach (using the Cholesky factor of $S^*S$, which may even fail), but in general it is not much better than the normal equations solver. 

However, if supplemented by a correction step, the seminormal solution becomes nearly as good as (sometimes even better than) the QR factorization based solver. 
The correction procedure, done in the same working precision, first computes the residual
\begin{equation}
r = \vec{\g} - S \vec{\bfalpha}
\end{equation} 
and then computes the correction $\delta \vec{\bfalpha}$  as 
\begin{equation}\label{eq:alpha:correction}
\delta \vec{\bfalpha}=R_S^{-1}(R_S^{-*}(S^* r))
\end{equation}
and the corrected solution $\vec{\bfalpha}_* = \vec{\bfalpha} + \delta \vec{\bfalpha}$. This process can be repeated. For an error analysis in favor of this scheme see \cite{BJORCK198731}. 

In an efficient  software implementation, we use the structure of $S$ and organize the data to increase locality, i.e. prefer matrix multiplications. 
A prototype of the computational scheme is given in Algorithm \ref{zd:ALG:Semi-Normal-LS}. (In lines 2. and 3. we give hints to an implementer, regarding high performance implementation based on the libraries \texttt{BLAS} and \texttt{LAPACK}.)
\begin{algorithm}[H]
	\caption{ Corrected semi-normal solution of  (\ref{zd:eq:g-S*alpha}) }
	\label{zd:ALG:Semi-Normal-LS}
	\begin{algorithmic}[1]
		\REQUIRE $R$, $\Lambda$, $G$, $S$ (Here applies the notation from \S \ref{SS=Normal_Eq_solution}.)
		\ENSURE Corrected solution $\vec{\bfalpha}_*$
		\STATE Compute the triangular factor $R_S$ in the QR factorization of $S$.
		\STATE $S^*\vec{\g} = [ (\overline{\Vanderm_{\ell,m}}\circ (R^* G))\eb ]$	\COMMENT{Use \texttt{xTRMM} from \texttt{BLAS 3}.}
        \STATE $\vec{\bfalpha}=R_S^{-1}(R_S^{-*}(S^*\vec{\g}))$\COMMENT{Use  \texttt{xTRSM} or \texttt{xTRTRS} or \texttt{xTRSV} from \texttt{LAPACK}.}
        \STATE $r_{\Box} = G - R \begin{pmatrix} \vec{\bfalpha} & \Lambda \vec{\bfalpha} & \Lambda^2 \vec{\bfalpha} & \ldots & \Lambda^{m-1}\vec{\bfalpha}\end{pmatrix} \equiv G - R \mathrm{diag}(\vec{\bfalpha}) \Vanderm_{\ell,m}$ 
        \STATE $S^* r = [ (\overline{\Vanderm_{\ell,m}}\circ (R^* r_{\Box}))\eb ]$ \COMMENT{Use \texttt{xTRMM} from \texttt{BLAS 3}.}
        \STATE $\delta \vec{\bfalpha}=R_S^{-1}(R_S^{-*}(S^* r))$ \COMMENT{Use  \texttt{xTRSM} or \texttt{xTRTRS} or \texttt{xTRSV} from \texttt{LAPACK}.}
        \STATE $\vec{\bfalpha}_* = \vec{\bfalpha} + \delta \vec{\bfalpha}$
	\end{algorithmic}
\end{algorithm}

It is certainly desirable to have an efficient QR factorization based LS solver that can exploit the particular structure of $S$ and thus lower the $O(m\ell^3)$ cost of a structure oblivious straightforward factorization. With such a factorization, one can solve the LS problem using (\ref{eq:RQg}) or the corrected semi-normal approach as in Algorithm \ref{zd:ALG:Semi-Normal-LS}.
 
The  recursive structure of the block rows of $S$ can indeed be exploited to compute its QR factorization at the cost of $O(\ell^3\log_2 m)$. In the next section we provide the details.

\subsection{QR factorization of $S$}\label{SS=QRF(S)}
\noindent We now provide the details of the new algorithm for computing the QR factorization of $S$.
The recursive structure of the algorithm is first described in \S \ref{SS=ALG-m=2p} for the simplest case of $m=2^p$, with the details of the kernel routine given in \S \ref{SSS-QR(A;B)}.  Section \ref{SS=ALG-general-m} provides the scheme with an arbitrary number of snapshots $m$. 

\subsubsection{The case $m=2^p$.}\label{SS=ALG-m=2p}
The basic idea is illustrated in (\ref{q:recursive:QR:S}) for $m=16$, and it is, in a sense, analogous to the FFT divide and conquer scheme.

\begin{algorithm}[H]
	\caption{Recursive QR factorization of $S$ in (\ref{zd:eq:g-S*alpha}) for $m=2^p$}
	\label{ALG:qr(S)-global}
	\begin{algorithmic}
	\REQUIRE Upper triangular $R\in\C^{\ell\times\ell}$; diagonal $\bfLambda\in\C^{\ell\times\ell}$; number of snapshots $m=2^p$
	\ENSURE Upper triangular QR  factor $R_S=T_p$ of $S\in\C^{2^p\ell \times \ell}$ in (\ref{zd:eq:g-S*alpha})
	\STATE
\begin{equation}\label{q:recursive:QR:S}
\boxed{\boxed{\begin{array}{c|r|r|r||r}
\boxed{\boxed{T_4}} & \longleftarrow \boxed{T_3} & \longleftarrow \boxed{T_2}    &\longleftarrow \boxed{T_1}          &\longleftarrow R\bfLambda^0 \\
\0  &\0 & \0       & \0           &\longleftarrow R\bfLambda^1 \\\hline
\0  & \0 & \0     & \longleftarrow T_1\bfLambda^2 & R\bfLambda^2 \\
\0 & \0 & \0           & \0           & R\bfLambda^3 \\\hline\hline
\0 & \0 & \longleftarrow T_2\bfLambda^4 & T_1\bfLambda^4 & R\bfLambda^4 \\
\0 & \0 & \0           & \0           & R\bfLambda^5 \\\hline
\0 & \0 & \0           & T_1\bfLambda^6 & R\bfLambda^6 \\
\0 & \0 & \0           & \0           & R\bfLambda^7 \\\hline\hline\hline
\0 & \longleftarrow T_3\bfLambda^8 & T_2\bfLambda^8 & T_1\bfLambda^8 & R\bfLambda^8 \\
\0 & \0 & \0             & \0           & R\bfLambda^9 \\\hline
\0 & \0 & \0             & T_1\bfLambda^{10} & R\bfLambda^{10} \\
\0 & \0 & \0             & \0              & R\bfLambda^{11} \\\hline\hline
\0 & \0 & T_2\bfLambda^{12}& T_1\bfLambda^{12} & R\bfLambda^{12} \\
\0 & \0 & \0             & \0              & R\bfLambda^{13} \\\hline
\0 & \0 & \0             & T_1\bfLambda^{14} & R\bfLambda^{14} \\
\0 & \0 & \0             & \0              & R\bfLambda^{15}
\end{array}}},\;\;\;
\boxed{
\begin{array}{ll}
\small{1}: & T_0 = R \cr
\small{2}: & \textbf{for}\;\; i=1:p \;\;\textbf{do}\cr 
3: & \begin{pmatrix} \boxed{T_i} \cr \0\end{pmatrix} = \texttt{qr}( \begin{pmatrix} \boxed{T_{i-1}} \cr T_{i-1}\bfLambda^{2^{i-1}} \end{pmatrix})\cr 
4: & \textbf{end for}
\end{array}
}
\end{equation}
\COMMENT{For a simple Matlab implementation see \S \ref{APP:QR-Khatri-RAO-VTR}.}
\end{algorithmic}
\end{algorithm}

Note that the core operation in the $i$th step of the above scheme is the QR factorization of the  $2\ell \times \ell$ structured matrix
\begin{equation}\label{eq:QR-structured}
\begin{pmatrix} 
T_{i-1} \cr T_{i-1}\bfLambda^{2^{i-1}} 
\end{pmatrix} = \left( \begin{smallmatrix} * & * & * & * \cr
0 & * & * & * \cr
0 & 0 & * & * \cr
0 & 0 & 0 & * \cr
\star & \star & \star & \star \cr
0 & \star & \star & \star \cr
0 & 0 & \star & \star \cr
0 & 0 & 0 & \star \cr
\end{smallmatrix}
\right),\;\;\mbox{i.e. implicitly of}\;\;\begin{pmatrix} 
T_{i-1} \cr \0_{(2^{i-1}-1)\ell,\ell} \cr T_{i-1}\bfLambda^{2^{i-1}} \cr \0_{(2^{i-1}-1)\ell,\ell} \end{pmatrix}
\end{equation}
and, thus, it can be computed more efficiently than in the general case of a $2\ell \times \ell$ dense matrix. 

\begin{remark}
{\em
	If we want to solve the LS problem using (\ref{eq:RQg}), then in Algorithm \ref{ALG:qr(S)-global} we need to include updating of $\vec{\g}$ for each $i=1,\ldots, p$. The algorithm then reads as follows: 
\begin{algorithmic}
\STATE $T_0=R$
\FOR {$i = 1 : p $}
\STATE $
\begin{pmatrix} \boxed{T_{i-1}} \cr T_{i-1}\bfLambda^{2^{i-1}} \end{pmatrix} = Q_i \boxed{T_i}
$
\COMMENT{Thin QR factorization. $Q_i$ is $2l\times l$.}
\FOR {$j = 0 : 2^i : 2^p - 2^i$}
\STATE $\vec{\g}(j\ell + 1 : (j+1)\ell) = Q_i^* \begin{pmatrix} \vec{\g}(j\ell + 1 : (j+1)\ell) \cr \vec{\g}((j+2^{i-1})\ell+1 : (j+2^{i-1}+1)\ell)\end{pmatrix}$
\ENDFOR
\ENDFOR
\end{algorithmic}
Note that after this sequence of updates, the array $\vec{\g}(1:\ell)$ contains the vector $Q_S^*\vec{\g}$ from (\ref{eq:RQg}); 
the least squares solution is then computed as the solution of the triangular system $T_p \vec{\bfalpha} = \vec{\g}(1:\ell)$.
}	
\end{remark}
\subsubsection{Details of the QR factorization of (\ref{eq:QR-structured})}\label{SSS-QR(A;B)}
The seemingly simple task to compute the QR factorization of two stacked triangular matrices (\ref{eq:QR-structured}) is an excellent case study for turning an algorithm into an efficient software implementation. 

Algorithm \ref{zd:ALG:Givens-QR} illustrates a straightforward scheme  to annihilate the $\ell(\ell+1)/2$ entries in the lower block in (\ref{eq:QR-structured}); it works for the general case of two independent upper triangular matrices stacked on top of each other and its cost is $\ell^3+O(\ell^2)$ \emph{flops}.

\begin{algorithm}[H]
	\caption{Givens QR factorization of the type (\ref{eq:QR-structured});}
	\label{zd:ALG:Givens-QR}
	\begin{algorithmic}[1]
		\REQUIRE Upper triangular matrices $A, B \in\C^{\ell\times \ell}$
\ENSURE The upper triangular factor in the QR factorization of $\begin{pmatrix} A\cr B\end{pmatrix} = \left( \begin{smallmatrix} * & * & * & * \cr
0 & * & * & * \cr
0 & 0 & * & * \cr
0 & 0 & 0 & * \cr
\star & \star & \star & \star \cr
0 & \star & \star & \star \cr
0 & 0 & \star & \star \cr
0 & 0 & 0 & \star \cr
\end{smallmatrix}\right)$
		\STATE $Q=\Id_{2\ell}$	
        \FOR {$j=1:\ell$}
        \FOR {$i=1:j$}
		\STATE Compute $2\times 2$ Givens rotation $\mathcal{G}$ such that 	$\mathcal{G}\begin{pmatrix} A(j,j)\cr B(i,j)\end{pmatrix}=\begin{pmatrix} \sqrt{|A(j,j)|^2+|B(i,j)|^2} \cr 0\end{pmatrix}$
		\STATE $\begin{pmatrix} A(j,j:\ell)\cr B(i,j:\ell)\end{pmatrix} := \mathcal{G} \begin{pmatrix} A(j,j:\ell)\cr B(i,j:\ell)\end{pmatrix} $; $( Q(:,i), Q(:,j) ) := ( Q(:,i), Q(:,j) )\mathcal{G}^* $
		\ENDFOR	
\ENDFOR
	\end{algorithmic}
\end{algorithm}
\noindent The nested loops in Algorithm \ref{zd:ALG:Givens-QR} access the matrix $B$ column-wise. Changing the loops into 
$$
\mathbf{for}\;\;i=1:\ell\;\; \{ \mathbf{for}\;\;j=i:\ell\;\; \{ 4:\; \ldots ; 5:\; \ldots \mathbf{end\; for}\}\;\mathbf{end\; for}\;\}
$$ 
will change the access to row-wise. The proper choice of loop ordering depends on the layout of the data matrices in the computer memory.

To enhance data locality, the annihilation strategy in Algorithm \ref{zd:ALG:Givens-QR} can be modified so that the communication between the arrays $A$ and $B$ is reduced to once per column -- it suffices to concentrate the mass of a column of $B$ into a single entry (e.g. using a single Householder reflector) and then move it up into the corresponding diagonal entry of $A$ using a single Givens rotation. 
This yields Algorithm \ref{zd:ALG:Householder-Givens-QR}.
\begin{algorithm}[H]
	\caption{Householder+Givens QR factorization of the type (\ref{eq:QR-structured});}
	\label{zd:ALG:Householder-Givens-QR}
	\begin{algorithmic}[1]
		\REQUIRE Upper triangular matrices $A, B \in\C^{\ell\times \ell}$
		\ENSURE The upper triangular factor in the QR factorization of $\begin{pmatrix} A\cr B\end{pmatrix} = \left( \begin{smallmatrix} * & * & * & * \cr
		0 & * & * & * \cr
		0 & 0 & * & * \cr
		0 & 0 & 0 & * \cr
		\star & \star & \star & \star \cr
		0 & \star & \star & \star \cr
		0 & 0 & \star & \star \cr
		0 & 0 & 0 & \star \cr
		\end{smallmatrix}\right)$
		%
		\FOR {$j=1:\ell$}
		\IF {$j>1$} 
		\STATE Compute Householder reflector $\mathcal{H} = I - \beta ww^*$ such that 
		$\mathcal{H} B(1:j,j) = \pm \| B(1:j,j)\|_2 e_1$
		\STATE $B(1,j) = \pm \| B(1:j,j)\|_2$
		\IF {$j<\ell$}
		\STATE Update $B$: $B(1:j,j+1:\ell) = B(1:j,j+1:\ell) - \beta w (w^* B(1:j,j+1:\ell))$
		\ENDIF
		\ENDIF
		\STATE Compute $2\times 2$ Givens rotation $\mathcal{G}$ such that 	$\mathcal{G}\begin{pmatrix} A(j,j)\cr B(1,j)\end{pmatrix}=\begin{pmatrix} \sqrt{|A(j,j)|^2+|B(1,j)|^2} \cr 0\end{pmatrix}$
		\STATE $\begin{pmatrix} A(j,j:\ell)\cr B(1,j:\ell)\end{pmatrix} := \mathcal{G} \begin{pmatrix} A(j,j:\ell)\cr B(1,j:\ell)\end{pmatrix} $; 
		\ENDFOR
	\end{algorithmic}
\end{algorithm}
The operations in Algorithm \ref{zd:ALG:Householder-Givens-QR} are illustrated in the scheme (\ref{eq:HG-QR-scheme}). 
\begin{eqnarray}
	&&
\left( \begin{smallmatrix} * & * & * & * \cr
0 & * & * & * \cr
0 & 0 & * & * \cr
0 & 0 & 0 & * \cr \hline \hline
{\star} & \star & \star & \star \cr
0 & \star & \star & \star \cr
0 & 0 & \star & \star \cr
0 & 0 & 0 & \star \cr
\end{smallmatrix}\right) \stackrel{1,1}{\longrightarrow}
\left( \begin{smallmatrix} 
\bullet & * & * & * \cr
0 & * & * & * \cr
0 & 0 & * & * \cr
0 & 0 & 0 & * \cr\hline \hline
\mathbf{0} & \star & \star & \star \cr
0 & \star & \star & \star \cr
0 & 0 & \star & \star \cr
0 & 0 & 0 & \star \cr
\end{smallmatrix}\right) \stackrel{1:2,2}{\Longrightarrow}
\left( \begin{smallmatrix} 
\bullet & * & * & * \cr
0 & * & * & * \cr
0 & 0 & * & * \cr
0 & 0 & 0 & * \cr\hline \hline
0 & \divideontimes & \star & \star \cr
0 & 0 & \star & \star \cr
0 & 0 & \star & \star \cr
0 & 0 & 0 & \star \cr
\end{smallmatrix}\right) \stackrel{1,2}{\longrightarrow}
\left( \begin{smallmatrix} 
\bullet & * & * & * \cr
0 & \bullet & * & * \cr
0 & 0 & * & * \cr
0 & 0 & 0 & * \cr\hline \hline
0 & 0 & \star & \star \cr
0 & 0 & \star & \star \cr
0 & 0 & \star & \star \cr
0 & 0 & 0 & \star \cr
\end{smallmatrix}\right) \stackrel{1:3,3}{\Longrightarrow}
\left( \begin{smallmatrix} 
\bullet & * & * & * \cr
0 & \bullet & * & * \cr
0 & 0 & * & * \cr
0 & 0 & 0 & * \cr\hline \hline
0 & 0 & \divideontimes & \star \cr
0 & 0 & 0 & \star \cr
0 & 0 & 0 & \star \cr
0 & 0 & 0 & \star \cr
\end{smallmatrix}\right)  \stackrel{1,3}{\longrightarrow}
\left( \begin{smallmatrix} 
\bullet & * & * & * \cr
0 & \bullet & * & * \cr
0 & 0 & \bullet & * \cr
0 & 0 & 0 & * \cr\hline \hline
0 & 0 & 0 & \star \cr
0 & 0 & 0 & \star \cr
0 & 0 & 0 & \star \cr
0 & 0 & 0 & \star \cr
\end{smallmatrix}\right) \stackrel{1:4,4}{\Longrightarrow} \nonumber \\
&& 
\left( \begin{smallmatrix} 
	\bullet & * & * & * \cr
	0 & \bullet & * & * \cr
	0 & 0 & \bullet & * \cr
	0 & 0 & 0 & * \cr\hline \hline
	0 & 0 & 0 & \divideontimes \cr
	0 & 0 & 0 & 0 \cr
	0 & 0 & 0 & 0 \cr
	0 & 0 & 0 & 0 \cr
\end{smallmatrix}\right)  \stackrel{1,4}{\longrightarrow}
\left( \begin{smallmatrix} 
	\bullet & * & * & * \cr
	0 & \bullet & * & * \cr
	0 & 0 & \bullet & * \cr
	0 & 0 & 0 & \bullet \cr\hline \hline
	0 & 0 & 0 & 0 \cr
	0 & 0 & 0 & 0 \cr
	0 & 0 & 0 & 0 \cr
	0 & 0 & 0 & 0 \cr
\end{smallmatrix}\right). \;\;\;\; 
\boxed{\mbox{Legend}:\ \begin{array}{l} 
 \stackrel{1,j}{\longrightarrow} \mbox{Givens rotation; Line 9.} \cr
 \stackrel{1:j,j}{\Longrightarrow} \mbox{Householder reflector; Lines 3--8.}
 \end{array}} \label{eq:HG-QR-scheme}
\end{eqnarray}
The above scheme is now a starting point for a block-oriented algorithm that delivers true high performance computation. Suppose our matrices are block partitioned with block size $b$ so that the total number of blocks is $\wp = \lceil \ell/b \rceil$; the leading $\wp-1$ diagonal blocks are $b\times b$, and the size of the last block is $(\ell - (\wp -1) b) \times (\ell - (\wp -1) b)$. Then we can simply imagine that e.g. in (\ref{eq:HG-QR-scheme}) each $*, \star, \bullet, \divideontimes, 0$ represents a $b\times b$ matrix (instead of being a scalar); the blocks in the last row and column may have one or both dimensions $(\ell - (\wp -1) b)$. For such a block partition, we use the notation $A[i,j]$, $B[i,j]$ to denote the submatrices (blocks) at the position $(i,j)$, and $A[i_1:i_2,j_1:j_2]$ is defined  analogously to the scalar case.

\begin{algorithm}[H]
	\caption{Block-oriented Householder+Givens QR factorization of the type (\ref{eq:QR-structured});}
	\label{zd:ALG:Householder-Givens-QR-blocks}
	\begin{algorithmic}[1]
		\REQUIRE Upper triangular matrices $A, B \in\C^{\ell\times \ell}$ and the block parameter $b$.
		\ENSURE The upper triangular factor in the QR factorization of $\begin{pmatrix} A\cr B\end{pmatrix} = \left( \begin{smallmatrix} * & * & * & * \cr
		0 & * & * & * \cr
		0 & 0 & * & * \cr
		0 & 0 & 0 & * \cr
		\star & \star & \star & \star \cr
		0 & \star & \star & \star \cr
		0 & 0 & \star & \star \cr
		0 & 0 & 0 & \star \cr
		\end{smallmatrix}\right)$
		%
		\STATE $\wp = \lceil \ell/b \rceil$; $b' = \ell - (\wp -1) b$. Introduce block partitions in $A$ and $B$.
		\FOR {$j=1:\wp$}
		\IF {$j>1$} 
		\STATE Compute the QR factorization $B[1:j,j]=\mathcal{H} \left( \begin{smallmatrix} \mathcal{R}\cr 0\end{smallmatrix}\right)$  of $B[1:j,j]$ as follows:
		\vspace{-2mm}
		\begin{itemize}
		\item[4.1:] The upper triangular factor $\mathcal{R}$ overwrites the leading submatrix of $B[1,j]$.
		\vspace{-2mm}	
		\item[4.2:] Write the accumulated product of  Householder reflectors $\mathcal{H}$ in the compact form $\mathcal{H}=\Id - \mathcal{W} \mathcal{T}\mathcal{W}^*$. (See Remark \ref{REM:block-reflector}.)
		\end{itemize}
		\vspace{-2mm}
		\IF {$j<\wp$}
		\STATE Update $B$: $B[1:j,j+1:\wp] = B[1:j,j+1:\wp] - \mathcal{W} \mathcal{T}^* (\mathcal{W}^* B[1:j,j+1:\wp])$
		\ENDIF
		\ENDIF
		\STATE Compute the QR factorization $\begin{pmatrix} A[j,j]\cr B[1,j]\end{pmatrix}= \mathcal{Q} \begin{pmatrix} \widehat{\mathcal{R}} \cr \0\end{pmatrix}$ so that the upper triangular factor $\widehat{\mathcal{R}}$ overwrites  $A[j,j]$. Use e.g. Algorithm \ref{zd:ALG:Householder-Givens-QR}.
        \IF {$j<\wp$}
		\STATE Update $A$ and $B$: $\begin{pmatrix} A[j,j+1:\wp]\cr B[1,j+1:\wp]\end{pmatrix} := \mathcal{Q}^* \begin{pmatrix} A[j,j+1:\wp]\cr B[1,j+1:\wp]\end{pmatrix} $; 
		\ENDIF
		\ENDFOR
	\end{algorithmic}
\end{algorithm}

\begin{remark}\label{REM:block-reflector}
{\em
In high performance libraries such as LAPACK, Householder reflectors are aggregated so that several of them can be applied more efficiently, with a better \emph{flop}-to-memory-reference ratio. If in Line 4 the matrix $B[1:j,j]$ has $b$ columns, then the QR factorization is achieved by a sequence of left multiplications $\mathcal{H}_b \cdots \mathcal{H}_2 \mathcal{H}_1 B[1:j,j]$, where $\mathcal{H}_i = \Id - \beta_i w_i w_i^*$ with $w_i(1:i-1)=\0$, $w_i(i)=1$. The compact form of $\mathcal{H} = \mathcal{H}_1 \mathcal{H}_2\cdots \mathcal{H}_b$ is then computed recursively as: $\mathcal{W}_1 = ( w_1 )$, $\mathcal{T}_1 = \beta_1$; $\mathcal{W}_j = ( \mathcal{W}_{j-1} \; w_j)$,
$$
\mathcal{T}_j = \begin{pmatrix} \mathcal{T}_{j-1} & -\beta_j \mathcal{T}_{j-1}\mathcal{W}_{j-1}^* w_j \cr \0 & \beta_j \end{pmatrix},\;\; j=2,\ldots, b .
$$	
For more details we refer to \cite{Schreiber-Van-Loan-WY-1989} and for a guidelines for an efficient implementation we suggest studying the structure of the subroutine \texttt{xGEQRF} in \texttt{LAPACK} -- essentially, the computation in Line 4 and Line 6 is already contained as a part of \texttt{xGEQRF}.
}	
\end{remark}
\subsubsection{The case of general $m$}\label{SS=ALG-general-m}
We now generalize the recursive scheme of  the Algorithm \ref{ALG:qr(S)-global}  to general dimension $m\neq 2^p$. 
First, introduce simple notation: $S$ will be considered as block-row partitioned with the $i$th block $S_{[i]}=R\Lambda^{i-1}$. The submatrix of $S$ consisting of consecutive blocks from the $i_1$th to the $i_2$th will be denoted by $S_{[i_1:i_2]}$. Note that 
$S_{[i_1:i_2]}=S_{[1:i_2-i_1+1]}\Lambda^{i_1-1}$.

As a motivation, note that the QR factorization of $S_{[1:16]}$ in the scheme (\ref{q:recursive:QR:S}) contains among its intermediate results the QR factorization also of e.g. $S_{[1:2]}$ (the factor is $T_1$), of $S_{[1:4]}$ (the factor is $T_2$), of $S_{[1:8]}$ (the factor is $T_3$), of $S_{[9:12]}$ (the factor is $T_2\Lambda^8$ as $S_{[9:12]}=S_{[1:4]}\Lambda^8$). Also, $T_i$ is the triangular factor of the leading $2^i$ block rows of $S$.

To exploit this, write $m$ as binary number $m\equiv \mathfrak{b}=(\mathfrak{b}_{\lfloor \log_2 m \rfloor},\ldots,\mathfrak{b}_1,\mathfrak{b}_0)_2$, i.e. 
\begin{equation}\label{eq:m-binary}
m  = \sum_{i=0}^{\lfloor \log_2 m \rfloor}\mathfrak{b}_i 2^i \equiv \sum_{j=1}^{j^*} 2^{i_j},\;\;\lfloor \log_2 m \rfloor=i_{j^*} > i_{j^*-1} > \cdots > i_2 > i_1 \geq 0 ,
\end{equation}
and introduce the block partition of $S$ as follows:
\begin{equation}\label{eq:S:block-partition}
S^T = \begin{pmatrix} S_{[1:2^{i_{j^*}}]}^T &  S_{[2^{i_{j^*}}+1:2^{i_{j^*}}+2^{i_{j^*-1}}]}^T & \ldots & 
S_{[2^{i_{j^*}}+\cdots+2^{i_2}+1 : 2^{i_{j^*}}+\cdots+2^{i_1}]}^T\end{pmatrix} .
\end{equation}
The QR factorization of the largest block $S_{[1:2^{i_{j^*}}]}$ can be computed by the $O(\lfloor \log_2 m \rfloor \ell^3)$ Algorithm \ref{ALG:qr(S)-global}, and the triangular factor of each of the subsequent blocks in (\ref{eq:S:block-partition}) is, up to column scaling by an appropriate power of $\Lambda$, available among the intermediate results, as discussed above. These are designated as \emph{local triangular factors}. This consists the first reduction step, that results in at most 
$\lfloor \log_2 m \rfloor +1$  local $\ell\times \ell$ upper triangular factors. 
In the second step, these are reduced, by building \emph{global triangular factors}, to a single triangular matrix at the cost of at most $O(\lfloor \log_2 m \rfloor \ell^3)$.

For an implementation of this procedure, it will be convenient to process the powers $2^{i_j}$ in an increasing order, i.e. to scan the binary representation $\mathfrak{b}$ from the right to the left. The local triangular factors of the blocks
\begin{equation}
S_{[1:2^{i_1}]}, S_{[2^{i_1}+1 : 2^{i_1}+2^{i_2}]}, S_{[2^{i_1}+2^{i_2}+1:2^{i_1}+2^{i_2}+2^{i_3}]}, \ldots, S_{[2^{i_1}+2^{i_2}+\cdots + 2^{i_{j^*-1}} + 1:2^{i_1}+2^{i_2}+2^{i_3}+\cdots + 2^{i_{j^*}}]}
\end{equation}
are computed by scaling the computed triangular factors of, respectively, 
\begin{equation}
S_{[1:2^{i_1}]}, S_{[1: 2^{i_2}]},  S_{[1: 2^{i_3}]}, \ldots , S_{[1: 2^{i_{j^*}}]}
\end{equation}
with, respectively,
\begin{equation}
\Lambda^0, \Lambda^{2^{i_1}}, \Lambda^{2^{i_1}+2^{i_2}}, \ldots , \Lambda^{2^{i_1}+2^{i_2}+\cdots + 2^{i_{j^*-1}}}
\end{equation}
and are built in into the global triangular factor by a sequence of updates. 
The procedure is summarized in Algorithm \ref{ALG:qr(S)-global-general-m}.

\begin{algorithm}[H]
	\caption{Recursive QR factorization of $S$ in (\ref{zd:eq:g-S*alpha})}
	\label{ALG:qr(S)-global-general-m}
	\begin{algorithmic}[1]
		\REQUIRE Upper triangular $R\in\C^{\ell\times\ell}$; diagonal $\bfLambda\in\C^{\ell\times\ell}$; number of snapshots $m$
		\ENSURE Upper triangular QR  factor $R_S=\mathbb{T}_{j-1}$ of $S$ in (\ref{zd:eq:g-S*alpha})
		\STATE Compute the binary representation (\ref{eq:m-binary}) of $m$: $m\equiv \mathfrak{b}=(\mathfrak{b}_{\lfloor \log_2 m \rfloor},\ldots,\mathfrak{b}_1,\mathfrak{b}_0)_2$ 
		\STATE Let $\lfloor \log_2 m \rfloor=i_{j^*} > i_{j^*-1} > \cdots > i_2 > i_1 \geq 0$ be as in (\ref{eq:m-binary})
		\STATE $T_0 = R$ 
		\IF{$i_1 = 0$}
		\STATE $\mathbb{T}_1 = T_0$; $j=2$; $\wp =1$
		\ELSE
		\STATE $\mathbb{T}_0=[]$; $j=1$; $\wp=0$
		\ENDIF
		\FOR {$k=1:i_{j^*}$} 
		\STATE $\begin{pmatrix} \boxed{T_k} \cr \0\end{pmatrix} = \texttt{qr}( \begin{pmatrix} \boxed{T_{k-1}} \cr T_{k-1}\bfLambda^{2^{k-1}} \end{pmatrix})$ \COMMENT{Local triangular factor. Use algorithms from \S \ref{SSS-QR(A;B)}.}
		\IF {$k=i_{j}$}
		\IF {$\mathbb{T}_{j-1} \neq []$}
		\STATE $\begin{pmatrix} \boxed{\mathbb{T}_j} \cr \0\end{pmatrix} = \texttt{qr}( \begin{pmatrix} \boxed{\mathbb{T}_{j-1}} \cr T_{k}\bfLambda^{\wp} \end{pmatrix})$ \COMMENT{Global triangular factor. Use algorithms from \S \ref{SSS-QR(A;B)}.}
		\ELSE 
		\STATE $\boxed{\mathbb{T}_j} = \boxed{T_k}$
		\ENDIF
		\STATE $j:= j+1$; $\wp := \wp + 2^k$
		\ENDIF
		\ENDFOR
		\STATE\COMMENT{For a simple Matlab implementation see \S \ref{APP:QR-Khatri-RAO-VTR}.}
	\end{algorithmic}
\end{algorithm}
\begin{remark}
	{\em
The above scheme can be easily adapted to work e.g. in base $3$, i.e. $m=3^p$ and, in general, using the representation of $m$ in base $3$. Such details/variations  become important for a custom made implementations on a  particular hardware.
}
\end{remark}
\subsection{Numerical stability}

For an efficient implementation, the computational scheme e.g. in Algorithm \ref{zd:ALG:Givens-QR} will be modified to enhance spatial and temporal locality of data, e.g. using tiling or blocking techniques as in Algorithm \ref{zd:ALG:Householder-Givens-QR-blocks}, or parallelized for multicore hardware.
It can be shown that with any such modification, the above computation is backward stable in the sense that the computed triangular factor is an exact factor of $S+\delta S$, where the backward error $\delta S$ is column-wise small, i.e. 
\begin{equation}\label{eq:dS(:,j)}
\|\delta S(:,j)\|_2 \leq \eta  \|S(:,j)\|_2, \;\;j=1,\ldots, \ell; \;\;\eta\leq f(\ell,m)\roff,
\end{equation}  
where $f(\ell,m)$ is modest polynomial that depends on the details of a particular implementation.\footnote{Note that (\ref{eq:dS(:,j)}) is much stronger statement than the usually used backward error bound $\|\delta S\|_F \leq g(\ell,m) \roff \|S\|_F$.}

This follows from the simple fact that in Algorithms \ref{ALG:qr(S)-global} and \ref{ALG:qr(S)-global-general-m} (using Algorithm \ref{zd:ALG:Givens-QR} with any ordering of Givens rotations, or Algorithms \ref{zd:ALG:Householder-Givens-QR} and \ref{zd:ALG:Householder-Givens-QR-blocks} as a kernel computational routine) we actually multiply the initial $S$ from the left by a sequence of elementary unitary matrices -- this is nothing else but independent unitary transformations (backward stable) of the columns of $S$. Hence, each column of $S$ has backward error that is small relative to that same column; in this way even the tiniest columns are preserved, independent of the remaining possibly much larger ones.

Due to (\ref{eq:dS(:,j)}), the corresponding  condition number that determines the accuracy of the decomposition is $\kappa_2(S_c)$, where $S_c$ is obtained from $S$ by scaling its columns so that $\|S_c(:,j)\|_2=1$ for all $j$. 
More precisely, if $\widetilde{R}_S = R_S + \delta R_S$ is the computed factor, then $\delta R_S=\Gamma R_S$ (i.e. $\delta R_S(:,j)=\Gamma(:,1:j) R_S(:,j)$, for all $j$) and, following \cite[\S 6.]{drmac-bujanovic-2008},
$$
\|\Gamma\|_F \leq \frac{\sqrt{ 8\ell}\eta}{1-\eta} \|S_c^{\dagger}\|_2 + O((\eta \|S_c^{\dagger}\|_2)^2) \leq \sqrt{ 8\ell}\eta \kappa_2(S_c) + O((\eta \|S_c^{\dagger}\|_2)^2).
$$
In terms of the initial data, the condition number of column-equilibrated $S$ is estimated as follows.
\begin{corollary}\label{COR=kappa(Sc)}
	 With the notation of Corollary \ref{COR:TM:Cond:A*B}, it holds that 
\begin{equation}\label{eq:kappa(Sc)}
	\kappa_2(S_c) = \sqrt{\kappa_2(C_s)}\leq \min ( \kappa_2(R_c), \kappa_2((\Vanderm_{\ell, m})_{r})) \leq \sqrt{\ell} \min ( \min_{D=\diag}\kappa_2(RD), \min_{D=\diag}\kappa_2(D \Vanderm_{\ell, m})).
\end{equation}
	\end{corollary}
\begin{proof}
It follows from Proposition \ref{PROP:S=KR} that $C_s = S_c^* S_c$; hence 
$\kappa_2(S_c)=\sqrt{\kappa_2(C_s)}$, and the first inequality follows from Corollary \ref{COR:TM:Cond:A*B}. The second inequality follows from the classical result \cite{slu-69}.	
	\end{proof}	
This  can be used to estimate the accuracy of Line 8 in Algorithm \ref{zd:ALG:Decision-Tree}: if $\min ( \kappa_2(R_c), \kappa_2((\Vanderm_{\ell, m})_{r}))$ in (\ref{eq:kappa(Sc)}) is below an appropriate threshold, certain level of accuracy can be guaranteed a priori.

\subsubsection{Importance of pivoting}\label{SSS=Pivoting-important}

For better accuracy of the solutions of triangular equations (forward and backward substitutions in (\ref{eq:alpha:new})), it would be advantageous to have column pivoted (rank revealing) QR factorization of $S$, i.e. that $R_S$ has strong diagonal dominance. Further, if $S$ is ill-conditioned, then the rank revealing QR factorization can be used to determine the numerical rank of $S$ and, by truncating $R_S$, to compute an approximate LS solution with certain level of sparsity.

Let $\mathrm{rank}(S)=r_S < \ell$, so that in the column pivoted QR factorization
\begin{equation}\label{eq:QRCP(S)}
S P = Q_S R_S = \begin{pmatrix} R_{[11]} & R_{[12]}\cr
\0 & \0\end{pmatrix} = Q_{S,r}\begin{pmatrix} R_{[11]} & R_{[12]}\end{pmatrix} , \;\;Q_{S,r} = Q_S(:,1:r).
\end{equation}	
Then the least squares problem $\|S\vec{\bfalpha}-\vec{\g}\|_2\rightarrow\min$ can be written as
\begin{eqnarray*}
	\|S \vec{\bfalpha}-\vec{\g}\|_2^2 &\!=\!& \| Q_{S,r} \begin{pmatrix} R_{[11]} & R_{[12]}\end{pmatrix}P^T \vec{\bfalpha} - Q_{S,r} Q_{S,r}^* \vec{\g} - (\Id -Q_{S,r} Q_{S,r}^*)\vec{\g}\|_2^2 \\
	&\!=\!& \| \begin{pmatrix} R_{[11]} & R_{[12]}\end{pmatrix} P^T \vec{\bfalpha}
	- Q_{S,r}^* \vec{\g}\|_2^2 + \| (\Id -Q_{S,r} Q_{S,r}^*)\vec{\g}\|_2^2\rightarrow\min,
\end{eqnarray*}
and the problem reduces to
\begin{equation}\label{eq:LS:reduced}
\| \begin{pmatrix} R_{[11]} & R_{[12]}\end{pmatrix} P^T \vec{\bfalpha}
- Q_{S,r}^* \vec{\g}\|_2 \longrightarrow \min_{\vec{\bfalpha}}.
\end{equation}
One particular vector in the solution manifold is 
\begin{equation}\label{eq:alpha:sparse}
\vec{\bfalpha}_{\diamond} = P \begin{pmatrix}R_{[11]}^{-1} Q_{S,r}^* \vec{\g} \cr \0\end{pmatrix}.
\end{equation} 
Note that $\vec{\bfalpha}_{\diamond}$ has at least $\ell-r_S$ zero entries, and that it is different from the shortest solution $\vec{\bfalpha}_{\ast}=S^{\dagger}\vec{\g}$. Hence, if the additional criterion is sparsity, the rank deficient least squares problem is best solved by (\ref{eq:QRCP(S)}, \ref{eq:alpha:sparse}). In the full rank case $\vec{\bfalpha}_{\diamond}=\vec{\bfalpha}_{\ast}$.
\begin{remark}
{\em
	Note that (\ref{eq:LS:reduced}) is under-determined and that we can add sparsity constraint analogously to (\ref{eq:LS-sparse-constraint}), which is accordance with Remark \ref{REM:DMDSP-Compr-Sens}. The sparsity of the explicit solution (\ref{eq:alpha:sparse}) is a good starting point for a quest for sparse solution.
}	
\end{remark}
\begin{remark}
	{\em
The Matlab backslash operator (e.g. \texttt{alpha=S\textbackslash g}) solves least squares problems using (\ref{eq:QRCP(S)}, \ref{eq:alpha:sparse}); the same procedure is used in \texttt{xGELSX} and \texttt{xGELSY} least squares solvers in \textsf{LAPACK}.
The pivoted QR factorization is computed by the \textsf{LAPACK} subroutine \texttt{xGEQP3}	(or \texttt{xGEQPF}). If sparsity is a desirable property of the solution in rank deficient case, then (\ref{eq:QRCP(S)}, \ref{eq:alpha:sparse}) should be preferred over using the Moore--Penrose pseudoinverse (e.g. \texttt{alpha=pinv(S)*g}, in Matlab).
	}
\end{remark}

The numerical rank $\widetilde{r}_S$ of $S$ is detected as follows: The column pivoted QR factorization computes
$SP=Q_S R_S$
and finds the smallest index $\widetilde{r}_S$ such that $|(R_S)_{\widetilde{r}_S+1,\widetilde{r}_S+1}|\leq \xi |(R_S)_{\widetilde{r}_S,\widetilde{r}_S}|$
where the threshold $\xi$ is usually $O(\ell\roff)$. Then, in the block partition 
\begin{equation}\label{eq:QRCP(S)-1}
S P = Q_S R_S = \begin{pmatrix} R_{[11]} & R_{[12]}\cr
\0 & R_{[22]}\end{pmatrix},\;\;R_{[11]}\in \C^{\widetilde{r}_S \times \widetilde{r}_S}
\end{equation}
it holds that
$\|R_{[22]}\|_F \leq \sqrt{\ell-\widetilde{r}_S} \xi |R_{\widetilde{r}_S,\widetilde{r}_S}|$, and $R_{[22]}$ it is set to zero, thus yielding a rank revealing decomposition of the form (\ref{eq:QRCP(S)}). The truncation of $R_{[22]}$ is justified by a backward perturbation of $S$. The choice of the threshold $\xi$ can be determined by taking into account the noise level on input, or it can be used to aggressively enforce low numerical rank (by allowing larger backward error) to obtain faster solver, or sparser least squares solution.

Indubitably, a QR factorization last squares algorithm should use pivoting, and it remains to see how to mount the pivoting device in Algorithms \ref{ALG:qr(S)-global} and \ref{ALG:qr(S)-global-general-m}.
The simplest way is to take the final upper triangular matrix $R_S$ on exit,  and recompute its QR factorization wit the Businger-Golub \cite{bus-gol-65} column pivoting; the more elegant one is to build the pivoting in the algorithm. It can be easily seen than the pivoting can be turned on at any (or every) stage in both algorithms; the permutations are accumulated/composed and pushed backward in the original matrix $S$. If the overhead due to pivoting at all stages is not acceptable, then we can settle with only the last step, when computing the last upper triangular matrix. 

\subsection{On real data and closedness under complex conjugacy}\label{SSS=Real-data-intro} 
If the data $\f_i$ in (\ref{eq:reconstruct-ell}) and the operator $\A$ are real, it is desirable that the reconstruction ansatz is a priori structurally real, even if the selected Ritz pairs $(z_j,\lambda_j)$ are in general complex. 
Further, in the real case, it is desirable to do the entire computation in real arithmetic, which would substantially improve the performances of the software. 

Keeping real arithmetic for real input (e.g. real matrices) and complex output that is a priori known to have complex conjugacy symmetry (eigenvalues and eigenvectors of real matrices) has important benefits with respect to numerical robustness (structure preserving that is  important from the point of view of the perturbation theory) and computational efficiency (real data structures and real arithmetic). These are exploited in the state of the art software for matrix computations.
So, for instance, in the \textsf{LAPACK} library, the driver routines for computing eigenvalues and eigenvectors $\texttt{xGEEV}$, $\texttt{xGEEVX}$ ($\texttt{x}\in\{ \texttt{S}, \texttt{D}\}$ for single and double precision real matrices) use only real arithmetic, and complex eigenvalues and eigenvectors are returned as ordered pairs of their real and imaginary parts. 

Hence, in a high performance \textsf{LAPACK}-based implementation of a minimizer for (\ref{eq:Omega:f-ZDa}), the computation with $(\lambda_j,z_j)$ should be in terms of $(\Re(\lambda_j),\Im(\lambda_j); \Re(z_j),\Im(z_j))$. Here $\Re(\cdot)$ and $\Im(\cdot)$ denote, respectively, the real and the imaginary part of a scalar, vector or matrix.

\subsubsection{Real reconstruction scheme of real data}
The following technical proposition provides all details needed for obtaining in real arithmetic a LS solution closed under complex conjugation, and, consequently, real approximants to the snapshots $\f_i$. 
\begin{proposition}
	Let all $\f_i$'s be real, i.e. $\f_i\in\R^n$, $\A\in\R^{n\times n}$.	
	If the wizard has selected the Ritz pairs so that with each complex pair $(z_j,\lambda_j)$ the sum on the right-hand side of (\ref{eq:reconstruct-ell}) contains also the contribution of its conjugate $(\overline{z_j},\overline{\lambda_j})$, then the corresponding coefficients are $\alpha_j$ and $\overline{\alpha_j}$, respectively. If $(z_j,\lambda_j)$ is real, then $\alpha_j$ is real as well. As a result, the computed approximation is real.
\end{proposition}
\begin{proof}
	Assume that the eigenvalues are so ordered that $\lambda_1,\ldots,\lambda_{\ell_1}$  are
	purely real, and the remaining complex eigenvalues are listed in groups of complex conjugate pairs $\lambda_j,
	\lambda_{j+1}=\overline{\lambda_j}$, with $\mathrm{Im}(\lambda_j)>0$. Let $\ell_2$ be the number of complex conjugate
	pairs. Hence $\ell=\ell_1 + 2\cdot \ell_2$, $\ell_1, \ell_2\geq 0$.
	
	Consider now a complex conjugate pair $(\lambda_j,z_j)$, $(\lambda_{j+1},z_{j+1})\equiv (\overline{\lambda_j},\overline{z_j})$. Since ${z}_j$ and $\overline{{z}_j}$ are linearly
	independent, the contribution of the directions of ${z}_j$ and
	${z}_{j+1}$ to  can be replaced by the span of the two purely
	real vectors $\Re({z}_j)$ and
	$\Im({z}_j)$. Let
	$$
	\Phi = \frac{1}{2}\begin{pmatrix}1 &  -\ii \cr 1 &
	\ii \end{pmatrix},\;\;\mbox{with}\;\; \Phi^{-1} = \begin{pmatrix} 1
	& 1 \cr \ii & -\ii \end{pmatrix}.
	$$
	Then $\begin{pmatrix} \Re({z}_j)\!\! &\!\!
	\Im({z}_j)\end{pmatrix} =
	\begin{pmatrix} {z}_j\!\! &\!\! {z}_{j+1} \end{pmatrix}\Phi$.
	Note that $\sqrt{2}\Phi$ is unitary and that for $j=\ell_1+1,\ell_1+3, \ldots, \ell-1$
	$$
	\Phi^{-1} \begin{pmatrix} \lambda_j & 0 \cr 0 & \overline{\lambda_j}\end{pmatrix} \Phi
	= \begin{pmatrix} \Re(\lambda_j) & \Im(\lambda_j) \cr
	-\Im(\lambda_j) & \Re(\lambda_j) \end{pmatrix} \equiv \hat{\Lambda}_j .
	\;\;\mbox{(Here $\Im(\lambda_j)>0$.)}
	$$
	Define
	block--diagonal matrix $\Phi_\lambda$ with unit diagonal $1\times 1$
	blocks $(1)$ for each real Ritz value $\lambda_j$ and $2\times 2$ matrix
	$\Phi$ for each complex conjugate pair $\lambda_j,\lambda_{j+1}$.
	Note that $Z_{\ell}\Phi_\lambda$ is now real matrix, 
	\begin{equation}\label{eq:Zl:real}
	Z_{\ell}\Phi_\lambda = \begin{pmatrix} z_1, &\ldots ,& z_{\ell_1}, & \Re(z_{\ell_1+1}), & \Im(z_{\ell_1+1}), & \ldots ,& \Re(z_{\ell-1}), & \Im(z_{\ell-1})\end{pmatrix} 
	\end{equation}
	and the real versions of the powers of $\bfLambda$ contain $2\times 2$ blocks for each pair $\lambda_j, \lambda_{j+1}=\overline{\lambda_j}$,
	\begin{equation}\label{eq:Lambda-i-real}
	\Phi_\lambda^{-1}\bfLambda^{i-1}\Phi_{\lambda} = \left[\bigoplus_{\mathrm{Im}(\lambda_j)=0}(\lambda_j^{i-1})\right] \bigoplus \left[\bigoplus_{\mathrm{Im}(\lambda_j)>0}\hat\Lambda_j^{i-1}\right] =
	\left(\begin{smallmatrix}
	\bullet &         &         &              &   &  &  \cr
	& \bullet &         &              &   &  &  \cr
	&         & \bullet &              &   &  &  \cr 
	&         &         & \bullet      & \bullet  &  &  \cr
	&         &         &  \bullet            & \bullet  &  &  \cr
	&         &         &              &   &  \bullet & \bullet \cr
	&         &         &              &   & \bullet  &    \bullet              
	\end{smallmatrix}\right)
	%
	%
	\equiv \widehat{\bfLambda}^{i-1}\in\R^{\ell\times\ell} .
	\end{equation}
	When used in the objective function, these transformations yield an equivalent formula
	\begin{equation}\label{eq:LS:real:1}
	\| (\WW\otimes \Id_n) \left[\begin{pmatrix} \f_1 \cr \vdots \cr \f_m\end{pmatrix} - \begin{pmatrix} Z_{\ell}\Delta_{\Lambda_1} \cr \vdots \cr Z_{\ell}\Delta_{\Lambda_m}\end{pmatrix} \vec{\bfalpha}\right]\|_2 = 
	\| \begin{pmatrix} \ww_1\f_1 \cr \vdots \cr \ww_m\f_m\end{pmatrix} - \begin{pmatrix} \ww_1 (Z_{\ell}\Phi_\lambda)(\Phi_\lambda^{-1}\bfLambda^0\Phi_\lambda) \cr \vdots \cr \ww_m (Z_{\ell}\Phi_\lambda)(\Phi_\lambda^{-1}\bfLambda^{m-1}\Phi_\lambda)\end{pmatrix} (\underbrace{\Phi_\lambda^{-1}\vec{\bfalpha}}_{\vec{\bfrho}})\|_2 .
	\end{equation}
	If we introduce a change of variables in (\ref{eq:LS:real:1}) by letting 
	$\vec{\bfrho} = \Phi_\lambda^{-1}\vec{\bfalpha}$, then the optimal $\vec{\bfrho}$ in (\ref{eq:LS:real:1}) must be real, and then the optimal $\vec{\bfalpha}=\Phi_{\lambda}\vec{\bfrho}$ reads
	\begin{equation}\label{eq:cc-alpha}
	\vec{\bfalpha} = \begin{pmatrix} \rho_1 & \ldots & \rho_{\ell_1}, & \rho_{\ell_1+1}+\ii\rho_{\ell_1+2}, & \rho_{\ell_1+1}-\ii\rho_{\ell_1+2}, & \ldots ,& \rho_{\ell-1}+\ii\rho_{\ell}, & \rho_{\ell-1}-\ii\rho_{\ell}\end{pmatrix}^T .
	\end{equation}
\end{proof}
The difference between (\ref{eq:LS:real:1}) and the original formulation (\ref{eq:Omega:f-ZDa}) is in replacing the diagonal scaling matrices $\bfLambda^{i-1}$ with block diagonal matrices $$
\widehat{\bfLambda}^{i-1}=\Phi_\lambda^{-1}\bfLambda^{i-1}\Phi_\lambda \equiv (\Phi_\lambda^{-1}\bfLambda\Phi_\lambda)^{i-1}
$$ 
containing $1\times 1$ or $2\times 2$ diagonal blocks. Similar holds for (\ref{eq:Omega:g-RDa}), where now we have real QR factorization $Z_{\ell}\Phi_\lambda = Q\left(\begin{smallmatrix} R \cr 0\end{smallmatrix}\right)$.
Note that $Z_{\ell}\Phi_\lambda$ is of same dimensions as $Z_{\ell}$, but it is a real matrix, thus occupying half the storage needed for $Z_{\ell}$. In a software implementation based e.g. on \textsf{LAPACK},  $Z_{\ell}\Phi_\lambda$ will be actually computed in the form (\ref{eq:Zl:real}), so that this switching to real arithmetic is simple.

Needles to say, computing the QR factorization in real arithmetic is (estimated, but can be easily demonstrated in numerical experiments with sufficiently large dimensions) at least twice faster than in complex. Moreover, as discussed above, in a high performance computing environment, we initially have $Z_{\ell}\Phi_\lambda$, and not its complexification $Z_{\ell}$. After the real QR factorization of $Z_{\ell}\Phi_{\lambda}$, the steps (\ref{eq:Omega:g-RDa}), (\ref{eq:Omega(alpha)}) are performed in real arithmetic.

And, finally, the fact that the solution is guaranteed to have complex conjugacy structure (\ref{eq:cc-alpha}) analogous to the one of the selected eigenpairs, thus yielding real approximation, makes the final argument in favor of this formulation. 
	
\subsubsection{Algorithms of \S \ref{SS=QRF(S)} for real data}
After rewriting the steps (\ref{eq:Omega:f-ZDa} -- \ref{zd:eq:g-S*alpha})
over $\R$ using (\ref{eq:Lambda-i-real}),  (\ref{eq:LS:real:1}),
it remains to adapt the algorithms described in \S \ref{SS=QRF(S)}  to the case of real $R$ and with the block--diagonal matrices $\widehat{\bfLambda}^{i-1}$ instead of the diagonal $\bfLambda^{i-1}$. 

Clearly, the global structure of Algorithm \ref{ALG:qr(S)-global} and Algorithm \ref{ALG:qr(S)-global-general-m} remains unchanged. The difference is in the structure (\ref{eq:QR-structured}), which is changed as illustrated below (three real eigenvalues ($\bullet$) and two complex conjugate pairs ($\blacklozenge$)): 
$$
T_{i-1}\bfLambda^{2^{i-1}} = \left(\begin{smallmatrix}
\ast &  \ast   &  \ast   &  \ast &  \ast & \ast & \ast \cr
     &  \ast   &  \ast   &  \ast &  \ast & \ast & \ast \cr
     &         &  \ast   &  \ast &  \ast & \ast & \ast \cr 
     &         &         &  \ast &  \ast & \ast & \ast \cr
     &         &         &       &  \ast & \ast & \ast \cr
     &         &         &       &       & \ast & \ast \cr
     &         &         &       &        &     & \ast
\end{smallmatrix}\right)
\left(\begin{smallmatrix}
\bullet &         &         &              &   &  &  \cr
        & \bullet &         &              &   &  &  \cr
        &         & \bullet &              &   &  &  \cr 
        &         &         & \blacklozenge     & \blacklozenge  &  &  \cr
        &         &         &  \blacklozenge    & \blacklozenge  &  &  \cr
        &         &         &              &   &  \blacklozenge & \blacklozenge \cr
        &         &         &              &   & \blacklozenge  &  \blacklozenge           
\end{smallmatrix}\right) = 
\left(\begin{smallmatrix}
\ast &  \ast   &  \ast   &  \ast &  \ast & \ast & \ast \cr
&  \ast   &  \ast   &  \ast &  \ast & \ast & \ast \cr
&         &  \ast   &  \ast &  \ast & \ast & \ast \cr 
&         &         &  \ast &  \ast & \ast & \ast \cr
&         &         &  \bigstar     &  \ast & \ast & \ast \cr
&         &         &       &       & \ast & \ast \cr
&         &         &       &        & \bigstar  & \ast
\end{smallmatrix}\right),\;\;\bullet, \ast, \blacklozenge, \bigstar \in\R.
$$
To put it simply, each complex conjugate pair of Ritz values creates a bulge ($\bigstar$).
Hence, before running algorithms described in \S \ref{SSS-QR(A;B)}, it suffices to apply Givens rotations to first annihilate the bulges ($\bigstar$), whose number equals the total number of complex conjugate pairs. It is easily seen that the rotations can be applied independently; there can be at most $\ell/2$ rotations, so the total cost of this correction is $O(\ell^2)$).
Such correction is needed only in the $(2,1)$ block, designated as $B$ in \S \ref{SSS-QR(A;B)}.

%

%% file: SOURCES2/LS_Concluding_Remarks.tex
\section{Concluding Remarks}
In this paper we have provided a new firm numerical linear algebra framework for solving structured least squares problem that arises in applications in e.g. computational fluid dynamic and multistatic antenna array processing. The numerical analysis explains the accuracy and the limitations of the normal equations based solution. New algorithm, based on a structure exploiting QR factorization has been presented with detailed numerical analysis and implementation details that should be the starting point for a high performance implementations on modern computing platforms. In the Appendix \S \ref{S=Matlab-sources}, we list few sample codes in Matlab; a \textsf{LAPACK} based implementation of all presented algorithms is under development.

%% file: SOURCES2/Section_Appendix_2.tex
\section{Matlab source codes}\label{S=Matlab-sources}
For the reader's convenience, in this section we provide Matlab implementations of some of the algorithms described in the paper.
\subsection{QR\_Khatri\_Rao\_VTR}\label{APP:QR-Khatri-RAO-VTR}
\begin{algorithm}[H]
	\caption{QR factorization of the Khatri-Rao product $\Vanderm_{\ell,m}^T \odot R$}
	\begin{lstlisting}
function T = QR_Khatri_Rao_VTR_2p( R, Lambda, p )
% QR_Khatri_Rao_VTR_2p computes the upper triangular factor in the QR factorization 
% of the Khatri-Rao product S=Khatri_Rao(Vlm.',R), where R is an <ell x ell> 
% upper triangular matrix, and Vlm is an  <ell x m> Vandermonde matrix V, whose 
% columns are defined as V(:,i) = Lambda.^(i-1), i = 1,...,m, and m=2^p.
% This code is written only to illustrate the global recursive structure of the
% algorithm. It has not been optimized for run-time efficiency.
%...............................................................................
% Coded by Zlatko Drmac, Department of Mathematics, University of Zagreb.
% drmac@math.hr
% April 2018
%...............................................................................
% Input:
% R      < ell x ell > (upper triangular) real or complex matrix
% Lambda < ell x   1 >  real or complex vector, defines Vlm = Vandermonde marix
% p      integer >=0    defines m = 2^p
% Output:
% T      < ell x ell > upper triangular QR fator of Khatri_Rao(Vlm.',R)
% ..............................................................................
%
T = R ; D = Lambda ;
%
for i = 1 : p
    [~, T] = qr( [ T ; T*diag(D)], 0 ) ; 
    D = D.^2 ; 
end
end	
\end{lstlisting}
\end{algorithm}

\begin{algorithm}[H]
	\caption{QR factorization of the Khatri-Rao product $\Vanderm_{\ell,m}^T \odot R$}
\begin{lstlisting}
function [ T, b1 ] = QR_Khatri_Rao_VTR_2p_b( R, Lambda, p, b )
% QR_Khatri_Rao_VTR computes the upper triangular factor in the QR factorization 
% of the Khatri-Rao product S=Khatri_Rao(Vlm.',R), where R is an <ell x ell> 
% upper triangular matrix, and Vlm is an  <ell x m> Vandermonde matrix V, whose 
% columns are defined as V(:,i) = Lambda.^(i-1), i = 1,...,m, and m=2^p.
% This code is written only to illustrate the global recursive structure of the
% algorithm. It has not been optimized for run-time efficiency.
%...............................................................................
% Coded by Zlatko Drmac, Department of Mathematics, University of Zagreb.
% drmac@math.hr
% April 2018
%...............................................................................
% Input:
% R      < ell x ell > (upper triangular) real or complex matrix
% Lambda < ell x 1 >    real or complex vector, defines Vlm = Vandermonde marix
% p      integer >=0  defines m = 2^p
% Output:
% T      < ell x ell > upper triangular QR fator of Khatri_Rao(Vlm.',R)
% ..............................................................................
%
T = R ; D = Lambda ; ell = length(Lambda) ; b1 = b ; 
%
for i = 1 : p
[Q, T] = qr( [ T ; T*diag(D)], 0 ) ; D = D.^2 ; s = 2^(i-1) ; 
for j = 0 : 2^i : 2^p - 2^i     
k = j + s ; 
b1( j*ell+1 : (j+1)*ell ) = Q' * [ b1( j*ell+1 : (j+1)*ell ) ; ...
b1( k*ell+1 : (k+1)*ell ) ] ;
end
end
end	
\end{lstlisting}
\end{algorithm}

\begin{algorithm}[H]
	\caption{QR factorization of the Khatri-Rao product $\Vanderm_{\ell,m}^T \odot R$}
	\begin{lstlisting}
function TT = QR_Khatri_Rao_VTR( R, Lambda, m )
% QR_Khatri_Rao_VTR computes the upper triangular factor in the QR
% factorization of the Khatri-Rao product S=Khatri_Rao(Vlm.',R), where
% R is an <ell x ell> upper triangular matrix, and Vlm is an  <ell x m> 
% Vandermonde matrix V, whose columns are defined as V(:,i) = Lambda.^i, 
% i = 0,...,m-1.
% This code has been written to illustrate the global recursive structure of the
% computation. It has not been optimized for speed. 
%...............................................................................
% Coded by Zlatko Drmac, Department of Mathematics, University of Zagreb.
% drmac@math.hr
% June 2018
%...............................................................................
% Input:
% R      <ell x ell> (upper triangular) matrix
% Lambda <ell x 1  >  vector, defines Vlm = Vandermonde marix
% m      <  1 x 1  >  integer > 0  
%...............................................................................
% Output:
% T      <ell x ell> upper triangular QR fator of Khatri_Rao(Vlm.',R)
%...............................................................................
%
J = floor(log2(m))+1 ; B = zeros(1,J);
for k = J : -1 : 1
    B(k) = mod(m,2) ; m = (m-B(k))/2 ; 
end
% B = (m)_2 is the binary representation of m
B  = fliplr(B)   ; % reverse the string and process it later from the left on
BB = find(B) - 1 ; % these are the increasingly ordered powers of two that 
                   % correspond to nonzero binary digits in (m)_2
lBB = length(BB) ; % this is the number of nonzero binary digits in (m)_2
% initialize the local and the global triangular factors
T  = R  ; D  = Lambda ;  D0 = D ;
if BB(1) == 0
   TT = T  ; k0 = 2  ; pow_shift = 1 ; 
else
   TT = [] ; k0 = 1  ; pow_shift = 0 ; 
end
% run over the powers of two and assemble the global factor
for k = 1 : BB(lBB)  % the main loop runs up to the highest power
    [~, T] = qr( [ T ; T*diag(D)], 0 ) ; % this is a local triangular factor
                                         % built for each binary digit in (m)_2
    if ( k == BB(k0) )   
       % interrupt for an update of the global triangular factor 
       if ~isempty(TT) 
           [~, TT] = qr( [ TT ; T*diag(D0.^(pow_shift)) ], 0 ) ;
       else % avoid trivial computation 
           TT = T ;
       end
      k0 = k0 + 1 ; pow_shift = pow_shift + 2^k ;      
    end 
    D = D.^2 ; 
end
%
end
	\end{lstlisting}
\end{algorithm}

\subsection{SLS\_Khatri\_Rao\_NE}\label{APP:SLS_Khatri_Rao_NE}
\begin{algorithm}[H]
	\caption{Normal equations solution to the LS problem $\|(\Vanderm_{\ell,m}^T \odot R) x - \vec{\g} \|_2\longrightarrow\min$}
	\begin{lstlisting}
function x = SLS_Khatri_Rao_NE( R, Lambda, G )
% SLS_Khatri_Rao_NE uses the normal equations to solve structured least squares 
% problem || S * x - vec(G) ||_2 --> min, where S is the Khatri-Rao product,
% S=Khatri_Rao(Vlm.',R), R is an <ell x ell> upper triangular matrix, and 
% Vlm is an  <ell x m> Vandermonde matrix, whose columns are defined as 
% Vlm(:,i) = Lambda.^i, i = 0,...,m-1. Lambda is complex vector.
% The code is writen for clarity, not optimality.
%...............................................................................
% Coded by Zlatko Drmac, Department of Mathematics, University of Zagreb.
% drmac@math.hr
%...............................................................................
% Input:
% R      <ell x ell> (upper triangular) matrix
% Lambda <ell x 1>    vector, defines Vlm = Vandermonde marix
% G      <ell x m>    real or complex data
% Output:
% x      <ell x 1>   the solution of the LS problem.
%
[ell,m] = size(G) ; 
Vlm  = ones( ell, m ) ; 
for j = 2 : m 
Vlm(:,j) = Lambda.*Vlm(:,j-1) ; 
end
K = (R'*R).*conj(Vlm*Vlm')        ; 
b = (conj(Vlm).*(R'*G))*ones(m,1) ;
x = K \ b                         ; 
end
	\end{lstlisting}
\end{algorithm}

\subsection{SLS\_W\_Khatri\_Rao\_NE}\label{APP:SLS_Khatri_Rao_NE}
\begin{algorithm}[H]
	\caption{ Normal equations solution to the LS problem\\ 
		$\|(W\otimes I_{\ell})\left[(\Vanderm_{\ell,m}^T \odot R) x - \vec{\g}\right] \|_2\longrightarrow\min$ }
	\begin{lstlisting}
function x = SLS_W_Khatri_Rao_NE( R, Lambda, G, W )
% SLS_Khatri_Rao_NE uses the normal equations to solve structured least squares 
% problem || kron(diag(W),eye(ell))(S * x - vec(G)) ||_2 --> min, where S is 
% the Khatri-Rao product, S=Khatri_Rao(Vlm.',R), R is an <ell x ell> upper 
% triangular matrix, and Vlm is an  <ell x m> Vandermonde matrix, whose columns
% are defined as Vlm(:,i) = Lambda.^i, i = 0,...,m-1. Lambda is complex vector.
% The code is writen for clarity, not optimality.
%...............................................................................
% Coded by Zlatko Drmac, Department of Mathematics, University of Zagreb.
% drmac@math.hr
%...............................................................................
% Input:
% R      <ell x ell> (upper triangular) matrix
% Lambda <ell x 1>    vector, defines Vlm = Vandermonde marix
% G      <ell x m>    real or complex data
% W      <m x 1>      real vector of positive weights
% Output:
% x      <ell x 1>   the solution of the LS problem.
%
[ell,m] = size(G) ; 
Vlm  = ones( ell, m ) ; 
for j = 2 : m 
Vlm(:,j) = Lambda.*Vlm(:,j-1) ; 
end
K = (R'*R).*conj(Vlm*diag(W.^2)*Vlm')             ; 
b = (conj(Vlm*diag(W)).*(R'*G*diag(W)))*ones(m,1) ;
x = K \ b                                         ; 
end
\end{lstlisting}
\end{algorithm}

\subsection{SLS\_Spectral\_Reconstruct\_W\_NE}

\begin{algorithm}[H]
\caption{Normal equations solution to the LS problem\\ 
	$\|\left[
	(\f_1, \f_2, \ldots, \f_m ) -		
	 Z_{\ell} \Delta_{\bfalpha} 
	 ( \Lambda_1,  \Lambda_2, \ldots, \Lambda_{m})
  \right] \WW\|_F \longrightarrow\min$}
\begin{lstlisting}
function x = SLS_Spectral_Reconstruct_W_NE( Z, Lambda, F, W )
% SLS_Spectral_Reconstruct_W_NE uses the normal equations to solve structured 
% least squares problem 
% ||( F-Z*diag(x)*(Lambda^0,Lambda^1,...,Lambda^(m-1) )*diag(W)||_F --> min, 
% where Lambda and x are complex vectors, W is real vector with positive 
% entries, and F and Z are real or complex matrices.
% The code is writen for clarity, not optimality.
%...............................................................................
% Coded by Zlatko Drmac, Department of Mathematics, University of Zagreb.
% drmac@math.hr
% April 2018
%...............................................................................
% Input:
% Z      <n x ell>    real or complex matrix
% Lambda <ell x 1>    real or complex vector, defines Vlm = Vandermonde marix
% F      <ell x m>    real or complex data
% W      <m x 1>      real vector of positive weights 
% Output:
% x      <ell x 1>   the solution of the LS problem.
%
[Q,R] = qr(Z,0) ;
G     = Q'*F    ;
[ell,m] = size(G)     ; 
Vlm  = ones( ell, m ) ; 
for j = 2 : m 
   Vlm(:,j) = Lambda.*Vlm(:,j-1) ; 
end
if nargin == 3 
   K = (R'*R).*conj(Vlm*Vlm')        ; 
   b = (conj(Vlm).*(R'*G))*ones(m,1) ;   
else
   K = (R'*R).*conj(Vlm*diag(W.^2)*Vlm')             ; 
   b = (conj(Vlm*diag(W)).*(R'*G*diag(W)))*ones(m,1) ;
end
x = K \ b ; 
end
\end{lstlisting}
\end{algorithm}

\subsection{Software license}
The following license applies to the software in this section.
\begin{verbatim}
%===============================================================================
% Copyright (c) 2018 AIMdyn Inc.  
% All right reserved. 
% 
% 3-Clause BSD License
% 
% Additional copyrights may follow.
% 
% 
% Redistribution and use in source and binary forms, with or without  
% modification, are permitted provided that the following conditions are met:
% 
% 1. Redistributions of source code must retain the above copyright notice,  
%    this list of conditions and the following disclaimer.
% 2. Redistributions in binary form must reproduce the above copyright notice, 
%    this list of conditions and the following disclaimer in the documentation 
%    and/or other materials provided with the distribution.
% 3. Neither the name of the copyright holder nor the names of its contributors 
%    may be used to endorse or promote products derived from this software  
%    without specific prior written permission.
% 
% The copyright holder provides no reassurances that the source code
% provided does not infringe any patent, copyright, or any other
% intellectual property rights of third parties.  The copyright holder
% disclaims any liability to any recipient for claims brought against
% recipient by any third party for infringement of that parties
% intellectual property rights.
% 
% THIS SOFTWARE IS PROVIDED BY THE COPYRIGHT HOLDER AND CONTRIBUTORS "AS IS" AND 
% ANY EXPRESS OR IMPLIED WARRANTIES, INCLUDING, BUT NOT LIMITED TO, THE IMPLIED 
% WARRANTIES OF MERCHANTABILITY AND FITNESS FOR A PARTICULAR PURPOSE ARE 
% DISCLAIMED. IN NO EVENT SHALL THE COPYRIGHT HOLDER OR CONTRIBUTORS BE LIABLE 
% FOR ANY DIRECT, INDIRECT, INCIDENTAL, SPECIAL, EXEMPLARY, OR CONSEQUENTIAL 
% DAMAGES (INCLUDING, BUT NOT LIMITED TO, PROCUREMENT OF SUBSTITUTE GOODS OR 
% SERVICES; LOSS OF USE, DATA, OR PROFITS; OR BUSINESS INTERRUPTION) HOWEVER 
% CAUSED AND ON ANY THEORY OF LIABILITY, WHETHER IN CONTRACT, STRICT LIABILITY, 
% OR TORT (INCLUDING NEGLIGENCE OR OTHERWISE) ARISING IN ANY WAY OUT OF THE USE 
% OF THIS SOFTWARE, EVEN IF ADVISED OF THE POSSIBILITY OF SUCH DAMAGE.
% 
%==============================================================================
\end{verbatim}